\documentclass[11pt,reqno]{amsart}
\usepackage{amsfonts,amssymb,amsmath,amsthm,adjustbox}
\usepackage[margin=1in]{geometry}
\usepackage[hidelinks]{hyperref}
\usepackage{enumerate}
\usepackage{microtype}
\usepackage{mathtools}
\usepackage{cite}
\usepackage{color}
\usepackage{tikz-cd}
\usepackage{float}
\usepackage{cleveref}

\def\Z{\mathbb{Z}}
\def\Q{\mathbb{Q}}

\def\F{\mathbb{F}}

\def\Gal{\operatorname{Gal}}
\def\GL{\operatorname{GL}}
\def\SL{\operatorname{SL}}

\def\det{\operatorname{det}}
\def\tr{\operatorname{tr}}
\def\im{\operatorname{im}}
\def\Frob{\operatorname{Frob}}
\def\Aut{\operatorname{Aut}}

\def\rad{\operatorname{rad}}

\def\lcm{\operatorname{lcm}}

\def\pico{\pi_{E_1, E_2}^{\mathrm{coprime}}}
\def\Cco{C_{E_1, E_2}^{\mathrm{coprime}}}
\def\CcoGeneric{C^{\mathrm{coprime}}}

\theoremstyle{plain}
\newtheorem{theorem}{Theorem}

\newtheorem{lemma}[theorem]{Lemma}
\newtheorem{proposition}[theorem]{Proposition}
\theoremstyle{definition}
\newtheorem{definition}[theorem]{Definition}

\newtheorem{example}[theorem]{Example}
\newtheorem{conjecture}[theorem]{Conjecture}

\newtheorem{question}[theorem]{Question}

\theoremstyle{remark}
\newtheorem{remark}[theorem]{Remark}

\title{On the Density of Coprime Reductions of Elliptic Curves}
\date{\today}

\author{Asimina S. Hamakiotes}
\address{Asimina S. Hamakiotes, Department of Mathematics, Fordham University, Lincoln Center, New York, NY 10023}

\author{Sung Min Lee}
\address{Sung Min Lee, Department of Mathematics, Wake Forest University, Winston-Salem, NC 27109}

\author{Jacob Mayle}
\address{Jacob Mayle, Department of Mathematical Sciences, University of Delaware, Newark, DE 19716}

\author{Tian Wang}
\address{Tian Wang, Department of Mathematics and Statistics, Concordia University, Montreal, QC H3B1B4}

\begin{document}

\begin{abstract} Given non-CM elliptic curves $E_1$ and $E_2$ over $\mathbb{Q}$, we study the natural density of primes $p$ of good reduction for which the orders of the groups $E_1(\F_p)$ and $E_2(\F_p)$ are coprime. This problem may be viewed as an elliptic curve analogue of the classical question concerning the density of coprime integer pairs. Motivated by Zywina's refinement of the Koblitz conjecture, we formulate a conjecture for the density of such primes. We prove that the series defining this constant converges and admits an almost Euler product expansion. In the case of Serre pairs, we give a closed formula for the constant and use it to prove a moments result describing the distribution of these constants as $(E_1, E_2)$ varies.
\end{abstract}

\maketitle

\section{Introduction}

A classical problem in elementary number theory asks for the probability that two positive integers chosen ``at random'' are coprime. In the sense of natural density, the answer is
\begin{align*}
    \prod_p \left(1-\frac{1}{p^2}\right)  =\frac{6}{\pi^2}.
\end{align*}
This result traces back to Euler's eighteenth century work on the Basel problem and the product expansion of the zeta function. In the nineteenth century, Dirichlet \cite{dirichlet1849} and Mertens \cite{Mertens1874} studied asymptotics of the Euler totient function, which gives rise to the density above. In this article, we propose and begin to study an elliptic curve analogue of this problem.

\begin{question}\label{coprimality}
    Given elliptic curves $E_1/\mathbb{Q}$ and $E_2/\mathbb{Q}$, what is the natural density of primes $p$ of good reduction for both $E_1$ and $E_2$ such that the orders of the groups $E_1(\F_p)$ and $E_2(\F_p)$ are coprime?
\end{question}

In the analogy with the classical problem, the role of ``random integers'' is played by the group orders $\#E_i(\F_p)$ for $i=1,2$ as $p$ varies. This analogy is supported by two observations. First, by the Hasse bound, the quantity $\#E(\F_{p})$ is comparable to $p$. Second, the divisibility properties of the sequence $\#E(\F_p)$ as $p$ varies exhibit similar statistical behavior to that of random integers. In particular, Cojocaru \cite{MR2076566} determined the density of primes $p$ of good reduction for which a fixed integer $m\ge 1$ divides $\#E(\F_p)$, and noted that this density should be roughly $1/m$ in the generic case.

Problems of this kind fit naturally into a long tradition of translating classical arithmetic questions into the setting of elliptic curves, which has led to notable applications in computational number theory and cryptography. On the computational side, Lenstra's elliptic curve method (ECM) \cite{MR916721} for integer factorization is an analogue of Pollard's $p-1$ method \cite{MR354514}. By replacing $(\mathbb{Z}/p\mathbb{Z})^\times$ with $E(\F_p)$, one replaces a smoothness condition on $p-1$ with a smoothness condition on $\#E(\F_p)$, and varying $E$ substantially increases the likelihood that this condition is satisfied.

From the cryptographic perspective, the security of widely deployed elliptic curve public-key systems \cite{MR0866109,Miller86} relies on the presumed difficulty of the discrete logarithm problem in $E(\F_p)$. Motivated by these applications, Koblitz \cite{MR0917870} initiated the study of the frequency with which $\#E(\F_p)$ is prime as $p$ varies. His original conjecture was later shown to be incorrect by an example of Jones \cite[Section 1.1]{MR2805578}, due to the presence of entanglements among the division fields of an elliptic curve. These issues were subsequently addressed by Zywina \cite{MR2805578}, who refined Koblitz's conjecture to account for such entanglements via the adelic Galois representation attached to $E$. While the refined Koblitz conjecture remains open, there has been progress in understanding the conjecture on average \cite{MR2843097,MR2805578,MR2534114}.

Related questions concern the frequency with which $\#E(\F_p)$ has a bounded number of prime factors. Murty and Miri \cite{MR1934487} established, under the Generalized Riemann Hypothesis (GRH), a lower bound for the number of primes $p$ for which $\#E(\F_p)$ has at most $16$ prime factors in the case of elliptic curves without complex multiplication (CM). This bound was subsequently improved to $8$ prime factors by Steuding and Weng \cite{MR2140162}. David and Wu \cite{MR2879973} later improved  the result under a weaker form of GRH. In addition, Cojocaru \cite{MR2167436} established an analogous result for CM elliptic curves unconditionally. 

In the present article, we study a distinct but related phenomenon: the coprimality of the group orders for reductions of two elliptic curves. Let $E_1/\mathbb{Q}$ and $E_2/\mathbb{Q}$ be non-CM elliptic curves, and for $x \ge 1$, define the counting function
\begin{equation} \label{E:pico}
\pico(x) \coloneqq \# \{ p \leq x : 
p \nmid N_{E_1} N_{E_2} \text{ and } \gcd(\# E_1(\F_p), \# E_2(\F_p)) = 1\},
\end{equation}
where $N_{E_i}$ denotes the conductor of $E_i$ for $i=1,2$. Our objective is to understand the asymptotic behavior of $\pico(x)$ as $x \to \infty$. 

Guided by Zywina's perspective (which builds on ideas of Lang and Trotter \cite{MR568299}), we conjecture an asymptotic for $\pico(x)$ in which the leading constant is explicit and arises from an inclusion-exclusion model built from the adelic Galois image of the abelian surface $E_1 \times E_2$.

\begin{conjecture}[Coprimality Conjecture] \label{C:Coprimality} Let $E_1$ and $E_2$ be non-CM elliptic curves over $\mathbb{Q}$ that are not $\overline{\Q}$-isogenous. Then
\[
\pico(x) \sim \Cco \cdot \frac{x}{\log x},
\] 
as $x \to \infty$, where $\Cco \geq 0$ is an explicit constant, defined in \eqref{E:Ccoprime}. If $\Cco = 0$, then we interpret the above asymptotic as stating that $\pico(x)$ is absolutely bounded as $x \to \infty$.
\end{conjecture}

The assumptions that both curves are non-CM and not $\overline{\Q}$-isogenous are needed in order for the adelic Galois image of $E_1 \times E_2$ to be ``large'' in an appropriate sense (see \Cref{SOIT_Product}). We do not attempt to formulate a version of \Cref{C:Coprimality} when at least one curve has CM, although we expect an analogous statement should be possible and intend to study this case in a future project. Unless stated otherwise, we assume throughout this article that $E_1$ and $E_2$ are non-CM elliptic curves.

\begin{remark} \label{Rem:NF}
Khai-Hoan Nguyen-Dang pointed out to us that  \Cref{C:Coprimality} can be formulated more broadly for non-CM elliptic curves $E_1, E_2$ defined over a number field $K$ that are not $\overline{K}$-isogenous. In this setting, \Cref{prop:Galois_image} is replaced by its number field analogue via  Lombardo's open image theorem for products of elliptic curves \cite[Theorem 1.1]{MR3515826} and other aspects carry over directly.
\end{remark}

Working with the adelic Galois image of $E_1 \times E_2$ is one of the main challenges and novelties of this article. Related work on the statistics of elliptic curve reduction (including on the Cyclicity \cite[pp.~465--468]{MR3223094} and Lang--Trotter \cite{MR568299} conjectures) considers Galois representations of a single curve. This setting is currently better understood than the product case, which is fundamental to the problem considered here. In addition to \emph{entanglements} within the division fields of a single curve, we must account for what we call \emph{entwinements}, namely interactions between the division fields attached to the two different curves (see \Cref{entwinement}).

We treat in detail the case where $(E_1, E_2)$ is a Serre pair, that is, when the image of the adelic Galois representation is as large as possible (see \Cref{S:SerrePairs}). This is the generic situation for a pair of elliptic curves, as proved in \cite{MR3071819} (which we state as \Cref{T:FewNonSerre} in this article). In particular, in \Cref{SerrePairConstant}, we give a closed formula for the coprimality constant $\Cco$ in this setting. Using this formula, we then prove the following moments result.

\begin{theorem}\label{averageresult} For any positive integer $t$, as $T \to \infty$,
$$\frac{1}{|\mathcal{E}(T)|}\sum_{(E_1,E_2) \in \mathcal{E}(T)} |\Cco-\CcoGeneric|^t \to 0, $$
where $\mathcal{E}(T)$ is defined in \eqref{E:ET} and the constant $\CcoGeneric$ is defined by
\begin{equation}\label{E:Cgeneric}
\CcoGeneric
\coloneqq \prod_{\ell} \left(1-\frac{(\ell+2)(\ell^2-\ell-1)}{(\ell-1)^3(\ell+1)^2}\right) \approx 0.39606.
\end{equation}
\end{theorem}

In particular, taking \(t=1\), \Cref{averageresult} shows that when pairs \((E_1,E_2)\) of elliptic curves are ordered by naive height as in \(\mathcal{E}(T)\), the average value of \(\Cco\) converges to \(\CcoGeneric\) as \(T\to\infty\). 

A further consequence of \Cref{SerrePairConstant} is the following.

\begin{theorem} \label{C:SerrePairBounds}
    If $(E_1, E_2)$ is a Serre pair, then
    \[ \frac{5014419112}{5014521525}\CcoGeneric \leq \Cco \leq \frac{1150648}{1118065} \CcoGeneric\]
    and both inequalities are sharp. In particular, $\Cco < 6/\pi^2$ for any Serre pair.
\end{theorem}

\begin{remark}
Numerically,
\[
\frac{5014419112}{5014521525} \CcoGeneric \approx 0.39606 \quad \text{ and } \quad 
\frac{1150648}{1118065}\CcoGeneric \approx 0.40761.
\]
Thus, assuming \Cref{C:Coprimality}, for any Serre pair \((E_1,E_2)\), the density of primes \(p\) for which
\(\gcd(\#E_1(\F_p),\#E_2(\F_p))=1\) is between approximately \(0.39606\) and \(0.40761\). In particular, it is strictly smaller than \(6/\pi^2 \approx 0.60793\), the probability that two positive integers chosen at random are coprime.
\end{remark}

\begin{remark}
Note that the constant $6/\pi^2$ also arises as the natural density of squarefree integers. This observation motivates us to look at the density $C^{\text{squarefree}}_E$ of primes $p$ for which  $\#E(\F_p)$ is squarefree. In \cite{MR2430992}, Cojocaru studied the asymptotic behavior of
\begin{equation}\label{eq:squarefree}
\#\{p \le x : p \nmid N_E, \, \#E(\mathbb{F}_p) \neq p+1 \text{ and is squarefree}\}
\end{equation}
 for elliptic curves  $E/\Q$ with CM by the full ring of integers of an imaginary quadratic field, obtained an explicit error term, and determined when $C^{\text{squarefree}}_E>0$ in this case.
 
 This problem was also investigated from a different perspective by Gekeler \cite{MR2429913}, who proved the ``probability" that a random elliptic curve $E$ over a random prime field $\F_p$ has squarefree group order $\#E(\F_p)$ is 
\[
C^{\text{squarefree}}=\prod_\ell\left( 1-\frac{\ell^3-\ell-1}{(\ell-1)^2\ell^2(\ell+1)}\right)\approx 0.44015.
\]
See \cite[p.56]{MR2429913} for the relevant definitions. 

Later, Akhtari et al. \cite{MR3204298} extended the study in \cite{MR2430992} to non-CM elliptic curves. In particular, they formulated a conjectural expression for the constant $C^{\text{squarefree}}_E$ for non-CM elliptic curves over $\Q$, and proved that, when averaging $C^{\text{squarefree}}_E$ over a certain two-parameter family of elliptic curves $E$, the result agrees with the constant  $C^{\text{squarefree}}$ obtained by Gekeler. 
\end{remark}

The paper is organized as follows. In \Cref{S:Prelims}, we review the preliminaries used throughout the article. In \Cref{S:Model}, we introduce the coprimality constant $\Cco$ and begin to study its properties. In \Cref{S:CountingMatrices}, we establish several matrix counting lemmas needed for the subsequent sections. In \Cref{S:CcoSP}, we specialize to the case of Serre pairs, obtain an explicit formula for $\Cco$ in this setting, and prove \Cref{C:SerrePairBounds}. In \Cref{S:AvgConst}, we put together our preceding results to prove \Cref{averageresult}. Finally, in \Cref{S:NumEx}, we present some numerical examples and give a practical criterion for determining whether a pair of elliptic curves is a Serre pair.

The GitHub repository \cite{CoprimeReductionsGitHub} accompanying this article is available at \smallskip

\centerline{\url{https://github.com/maylejacobj/CoprimeReduction}} \smallskip

\noindent All computations were carried out using \texttt{Magma V2.28-21}.

\subsection*{Acknowledgments}
We thank Jeremy Rouse for sharing constructive feedback related to this article. We also thank Nguyen-Dang Khai-Hoan for bringing to our attention a  number field generalization of \Cref{C:Coprimality} discussed in Remark \ref{Rem:NF}.

\section{Preliminaries} \label{S:Prelims}

\subsection{Fiber Products of Groups}

We begin by recalling some basic facts about fiber products, which we will use repeatedly to describe Galois images for products of elliptic curves. 

Let $G_1$, $G_2$, and $Q$ be groups. Let $\phi_1 \colon G_1 \to Q$ and $\phi_2 \colon G_2 \to Q$ be surjective homomorphisms. The \textit{fiber product} of $G_1$ and $G_2$ over $(\phi_1, \phi_2)$ is the subgroup of $G_1 \times G_2$ defined by
\[
G_1\times_{(\phi_1, \phi_2)} G_2 \coloneqq\{(\gamma_1, \gamma_2)\in G_1\times G_2: \phi_1(\gamma_1)=\phi_2(\gamma_2)\}.
\]
Equivalently, $G_1\times_{(\phi_1, \phi_2)} G_2$ is the preimage of the diagonal subgroup of $Q \times Q$ under the product map $\phi_1 \times \phi_2 \colon G_1 \times G_2 \to Q \times Q$. By the first isomorphism theorem, the order formula
\begin{equation}\label{eq:fiber_product}
    |G_1 \times_{(\phi_1, \phi_2)} G_2| = \frac{|G_1| |G_2|}{|Q|}
\end{equation}
follows provided that $G_1$ and $G_2$ are finite.

By construction, $G_1\times_{(\phi_1, \phi_2)} G_2$ is a subgroup of $G_1\times G_2$ whose projections to $G_1$ and $G_2$ are surjective. Conversely, every subgroup of $G_1\times G_2$ with surjective projections arises in this way by Goursat's lemma \cite[p.~75]{MR1878556}, which we now recall.

\begin{lemma} \label{L:Goursat}
    Let $G\subseteq G_1\times G_2$ be a subgroup. The projections of $G$ onto $G_1$ and $G_2$ are surjective if and only if there exist a group $Q$ and surjective homomorphisms $\phi_i \colon G_i \to Q$ for $i=1,2$ such that 
    \[ G = G_1\times_{(\phi_1, \phi_2)} G_2.\] 
\end{lemma}

We will also need a bound on the index of a fiber product of subgroups.
\begin{lemma}\label{lem:index_lemma}
    Let $G_1$ and $G_2$ be finite groups, and let $\phi_i \colon G_i \to Q$ be surjective homomorphisms onto a common finite group $Q$. Let $H_1$ and $H_2$ be subgroups of $G_1$ and $G_2$. Suppose that $\phi_1(H_1) = \phi_2(H_2) \eqqcolon Q'$. For $i=1,2$, let $\phi_i|_{H_i} \colon H_i \to Q'$ be the restriction of $\phi_i$ to $H_i$ with codomain $Q'$. Then 
    \[
    [G_1 \times_{(\phi_1, \phi_2)} G_2 : H_1 \times_{(\phi_1|_{H_1}, \phi_2|_{H_2})} H_2] \leq  [G_1 : H_1] [G_2 : H_2],
    \]
    with equality if and only if $Q'=Q$. 
\end{lemma}
\begin{proof}
    Applying \eqref{eq:fiber_product} to both fiber products gives
    \[
    |G_1 \times_{(\phi_1, \phi_2)} G_2| = \frac{|G_1| |G_2|}{|Q|} \quad \text{ and }\quad  |H_1 \times_{(\phi_1|_{H_1}, \phi_2|_{H_2})} H_2| = \frac{|H_1| |H_2|}{|Q'|}.
    \]
    Therefore,
    \[
    [G_1 \times_{(\phi_1, \phi_2)} G_2 : H_1 \times_{(\phi_1|_{H_1}, \phi_2|_{H_2})} H_2] = \frac{|G_1 \times_{(\phi_1, \phi_2)} G_2|}{|H_1 \times_{(\phi_1|_{H_1}, \phi_2|_{H_2})} H_2|} = \frac{[G_1 : H_1] [G_2 : H_2]}{[Q : Q']}.
    \]
    Since $[Q : Q'] \geq 1$, the desired inequality follows, and equality holds precisely when $Q'=Q$.
\end{proof}

\subsection{Galois Representations of Elliptic Curves}
Let \( E/\Q \) be an elliptic curve. For any positive integer $n$, write $E[n]$ for the $n$-torsion subgroup of $E(\overline{\Q})$. The \emph{adelic Tate module} of $E$ is
\[
T(E) \coloneqq \varprojlim_n E[n],
\]
where the inverse limit is taken over the directed system ordered by divisibility. It is a free $\widehat{\Z}$-module of rank $2$, where $\widehat{\Z}$ denotes the ring of profinite integers. The absolute Galois group \(  \Gal(\overline{\Q}/\Q) \) acts naturally on each \( E[n] \) and hence on \( T(E) \).
Choosing a $\widehat{\Z}$-basis of $T(E)$ induces, for each $n$, a $\Z/n\Z$-basis of $E[n]$, and hence the \emph{mod \( n \)} and \emph{adelic Galois representations} of $E$,
\[
\rho_{E,n} \colon \Gal(\overline{\Q}/\Q)  \to \Aut(E[n]) \simeq \GL_2(\Z/n\Z)
\quad \text{and} \quad
\rho_E \colon \Gal(\overline{\Q}/\Q) \to \Aut(T(E)) \simeq \GL_2(\widehat{\Z}).
\]
These representations are compatible in the sense that for each positive $n$,
\[
\rho_{E, n} = \pi_n \circ \rho_E,
\]
where $\pi_n \colon \GL_2(\widehat{\Z}) \to \GL_2(\Z/n\Z)$ denotes the reduction modulo $n$ map. The images of these representations are well defined up to conjugacy, and we denote them by  \( G_E(n) \) and \( G_E \), respectively.

In this article, we will be primarily concerned with Galois representations attached to non-CM elliptic curves. In this setting, Serre proved that the adelic image is ``large'' in the following sense.

\begin{theorem}[Serre, \protect{\cite[Th\'eor\`eme 3]{MR0387283}}] \label{SOIT}
Let \( E/\Q \) be a non-CM elliptic curve. Then \( G_E \) is an open subgroup of \( \GL_2(\widehat{\Z}) \), and consequently the adelic index $[\GL_2(\widehat{\Z}) : G_E]$ is finite.
\end{theorem}

This theorem guarantees the existence of a positive integer $m$ such that
\[
G_E = \pi_m^{-1}\left(G_E(m)\right),
\]
where $\pi_m$ again denotes the reduction modulo $m$ map (see, for example, \cite{Brau}).  The smallest such integer \( m \)
is called the \emph{adelic level} of \( E \) and is denoted by \( m_E \). Every prime $\ell$ for which the mod $\ell$ (or $\ell$-adic) Galois representation of $E$ is not surjective divides $m_E$; however, the converse does not hold in general because of the possibility of entanglements among division fields (see \Cref{S:SerreCurves}). The adelic level is a crucial invariant for understanding $G_E$. It has also appeared in the literature under the names
\emph{torsion conductor} and \emph{image conductor}, though the term adelic level appears to have become more standard recently (for instance, this is the term used in the LMFDB).

Let $p$ be a prime of good reduction for $E$ (equivalently, a prime not dividing the conductor $N_E$ of $E$). The trace of Frobenius is the integer $a_p(E)$ characterized by
\[ \#E(\F_p) = p + 1 - a_p(E). \]
By the Hasse bound, we know that $|a_p(E)| \leq 2 \sqrt{p}$. Let $\Frob_p \in \Gal(\overline{\Q}/\Q)$ be a Frobenius element at $p$, which is well defined up to conjugacy. If $n$ is any integer coprime to $p$, then
\begin{equation} \label{E:TrDet}
    \tr(\rho_{E,n}(\Frob_p)) \equiv a_p(E) \pmod n \quad \text{and} \quad \det(\rho_{E,n}(\Frob_p)) \equiv p \pmod{n}.
\end{equation}

Let $\chi \colon \Gal(\overline{\Q}/\Q) \to \widehat{\Z}^\times$ denote the adelic cyclotomic character. The Weil pairing implies that
\begin{equation} \label{E:Weil} \det \circ \rho_E = \chi. \end{equation}
In particular,
\( \det G_E = \chi(\Gal(\overline{\Q}/\Q)) = \widehat{\Z}^\times \)
and hence for every positive integer $n$,
\begin{equation} \label{E:fulldet} \det G_E(n) = (\Z/n\Z)^\times.\end{equation}

\subsection{Serre Curves} \label{S:SerreCurves} We now turn to the class of elliptic curves known as Serre curves, characterized by having adelic Galois image as large as possible. This class is generic in the sense that it has density $1$ among all elliptic curves over $\Q$, when ordered by naive height, as shown by Jones \cite{MR2563740}.

Let $E/\Q$ be a non-CM elliptic curve. Write a factored Weierstrass equation for $E$,
$$E \colon y^2 = (x-e_1)(x-e_2)(x-e_3),$$
with $e_1, e_2, e_3 \in \overline{\Q}$. Then the 2-torsion subgroup of $E(\overline{\Q})$ is
$$E[2] = \{\mathcal{O}, (e_1,0), (e_2,0),(e_3,0)\} \simeq \Z/2\Z \oplus \Z/2\Z$$
and the discriminant of this Weierstrass equation is
\begin{equation} \label{E:DeltaE}
    \Delta_E = 16 [(e_1-e_2)(e_2-e_3)(e_3-e_1)]^2.
\end{equation}
Let $\Delta' \in \mathbb{Z}$ denote the squarefree part of $\Delta_E$. Note that $\Delta'$ is independent of the chosen model for $E$, since changing the Weierstrass equation multiplies $\Delta_E$ by a $12$th power in $\mathbb{Q}^\times$.

If $\Delta_E$ is a rational square (equivalently, if $\Delta' = 1$), then the adelic index $[\GL_2(\widehat{\Z}) : G_E]$ is at least $12$ by \cite[Proposition 2.14]{MR4732685}. We therefore assume that $\Delta_E$ is not a rational square, and set
\[ K \coloneqq\Q(\sqrt{\Delta_E}) = \Q(\sqrt{\Delta'}). \]
From \eqref{E:DeltaE}, we see that for any $\sigma \in \Gal(\overline{\Q}/\Q)$, the sign of $\sigma(\sqrt{\Delta_E})$ is determined by the parity of the permutation that $\sigma$ induces on the roots $e_1, e_2,$ and $e_3$.  Concretely,
\begin{equation} \label{E:epsilon_map}
\sigma(\sqrt{\Delta_E}) = \epsilon(\rho_{E,2}(\sigma)) \sqrt{\Delta_E},
\end{equation}
where we use the identification $\GL_2(\Z/2\Z) \simeq S_3$ and  $\epsilon \colon S_3 \to \{\pm 1\}$ is the sign character. 

Let $d_E$ be the conductor of $K$, that is, the smallest positive integer such that $K \subseteq \Q(\zeta_{d_E})$; such an integer must exist by the Kronecker--Weber theorem. For a quadratic number field, the conductor equals the absolute value of the field discriminant, so
$$d_E = \begin{cases}
    |\Delta'| & \text{ if } \Delta' \equiv 1 \pmod 4, \\
    4|\Delta'| & \text{ otherwise}.
\end{cases}$$
Let $\chi_{d_E} \colon \Gal(\overline{\Q} / \Q) \to (\Z/d_E\Z)^\times$ denote the mod $d_E$ cyclotomic character. Then there exists a unique quadratic character $\alpha : (\Z/d_E\Z)^\times \to \{\pm 1\}$ for which
\begin{equation}\label{E:kronecker_map}
    \sigma(\sqrt{\Delta_E}) = \alpha( \chi_{d_E} (\sigma))\sqrt{\Delta_E} = (\alpha \circ \det)( \rho_{E,d_E}(\sigma))\sqrt{\Delta_E},
\end{equation}
where the second equality follows from \eqref{E:Weil}. Comparing \eqref{E:epsilon_map} and \eqref{E:kronecker_map}, we obtain the relation
\begin{equation} \label{SerreEntanglement}
    \epsilon(\rho_{E,2}(\sigma)) = (\alpha \circ \det)(\rho_{E,d_E}(\sigma)). 
\end{equation}

Let $M_E \coloneqq \lcm(2,d_E)$. For any integer $r$ dividing $M_E$ and matrix $M \in \GL_2(\Z/M_E \Z)$, write $M_r$ for the reduction of $M$ modulo $r$. We define
\begin{align*}
     H_E(M_E) \coloneqq \left\{M \in \GL_2(\Z/M_E\Z) : \epsilon(M_2) = (\alpha \circ \det)(M_{d_E})\right\}.
\end{align*}
Let $H_E$ denote the full preimage of $H_E(M_E)$ in $\GL_2(\widehat{\Z})$. We have that $H_E(M_E)$ is an index $2$ subgroup of $\GL_2(\Z/M_E\Z)$, and \eqref{SerreEntanglement} implies $G_E(M_E) \subseteq H_E(M_E)$. Hence
$$[\GL_2(\widehat{\Z}):G_E] \geq [\GL_2(\Z/M_E\Z) : G_E(M_E)] \geq 2.$$
When $[\GL_2(\widehat{\Z}):G_E] = 2$, we say that $E$ is a \emph{Serre curve}.

For Serre curves, the adelic level and the mod $m$ images admit simple explicit descriptions.
\begin{proposition}\label{adeliclevelofSerrecurve}
Let \( E/\Q \) be a Serre curve with discriminant \( \Delta_E \), and let \( \Delta' \) denote the squarefree part of \( \Delta_E \). Then
\[
m_E =
\begin{cases}
2|\Delta'|, & \text{if } \Delta' \equiv 1 \pmod{4}, \\
4|\Delta'|, & \text{otherwise}.
\end{cases}
\]
\end{proposition}

\begin{proof}
See \cite[pp.~696--697]{MR2534114}.
\end{proof}

The following lemma provides a lower bound for the adelic level $m_E$, which will be used in the proof of \Cref{C:SerrePairBounds}.

\begin{lemma}\label{Atleast6}
    If $E/\Q$ is a Serre curve, then $m_E \geq 6$.
\end{lemma}
\begin{proof}
    By \Cref{adeliclevelofSerrecurve}, it suffices to rule out the cases $\Delta' = \pm 1$. As mentioned above, if $\Delta' = 1$, then $[\GL_2(\widehat{\Z}):G_E] \geq 12$ by \cite[Proposition 2.14]{MR4732685}, so $E$ cannot be a Serre curve. Now suppose that $\Delta' = -1$. By \cite[Theorem (2)]{MR2995149}, the representation $\rho_{E,4}$ is not surjective, so $E$ cannot be a Serre curve by \cite[Theorem 1.6]{MR3447646} (see also \cite[Theorem 1.8]{MR3349445}).
\end{proof}

For a Serre curve, $G_E$ is the full preimage of $H_E(M_E)$ in $\GL_2(\widehat{\Z})$. This leads to the following description of the mod $m$ images. In the statement of the proposition, we write $H_E(m)$ for the image of $H_E$ under the reduction map $\GL_2(\widehat{\Z}) \to \GL_2(\Z/m\Z)$.

\begin{proposition}\label{SerreCurveGroup}
Let \( E/\Q \) be a Serre curve. For any positive integer \( m \), we have
\[
G_E(m) =
\begin{cases}
\GL_2(\Z/m\Z), & \text{if } m_E \nmid m, \\
H_E(m), & \text{if } m_E \mid m.
\end{cases}
\]
In particular, the order of $G_E(m)$ is $|\GL_2(\Z/m\Z)|$ when $m_E \nmid m$ and $\frac{1}{2}|\GL_2(\Z/m\Z)|$ when $m_E \mid m$.
\end{proposition}
\begin{proof}
See \cite[Proposition~2.4]{Lee_Mayle_Wang_2025}.
\end{proof}

For a positive integer $m$, Jones \cite{MR2534114} expresses $H_E(m)$ as the kernel of a quadratic character $\psi_m$ on $\GL_2(\Z/m\Z)$, which can be written as a product of local characters,
\begin{equation*}
\psi_m=\prod_{\ell^k\parallel m}\psi_{\ell^k},
\end{equation*}
where each \(\psi_{\ell^k}\) is given explicitly in \cite[Sec.~4]{MR2534114}. We will use in particular that if $E$ is a Serre curve with $m_E$ squarefree, then 
\(\psi_{\ell}=\left(\frac{\det(\cdot)}{\ell}\right)\) for odd primes \(\ell\) dividing $m_E$ and \(\psi_{2}=\epsilon \) is the sign map.

\subsection{Galois Representations for Products of Elliptic Curves} \label{SubSec:GalProd}

Let $E_1$ and $E_2$ be elliptic curves over $\Q$.  As in the single curve case,  the adelic Tate module of the product $E_1\times E_2$ is defined as the inverse limit of the $n$-torsion subgroups of $E_1\times E_2$, ordered by divisibility. We have that
\[
T(E_1 \times E_2) = T(E_1) \oplus T(E_2).
\]
The natural action of $\Gal(\overline{\Q}/\Q)$ on $T(E_1 \times E_2)$ respects this decomposition, and gives rise to the adelic Galois representation
\[
\rho_{E_1 \times E_2} \colon \Gal(\overline{\Q}/\Q) \longrightarrow \Aut(T(E_1))  \times \Aut(T(E_2)) \simeq \GL_2(\widehat{\Z}) \times \GL_2(\widehat{\Z}).
\]
Likewise, for each positive integer $n$,  we have the mod $n$ Galois representation
\[
\rho_{E_1 \times E_2, n} \colon \Gal(\overline{\Q}/\Q) \longrightarrow \Aut(E_1[n]) \times \Aut(E_2[n]) \simeq \GL_2(\Z/n\Z) \times \GL_2(\Z/n\Z).
\]
We denote the images of these representations by
\[
G_{E_1 \times E_2} = \im \rho_{E_1 \times E_2}
\quad \text{and} \quad
G_{E_1 \times E_2}(n) = \im \rho_{E_1 \times E_2, n}.
\]

By \eqref{E:Weil}, we have
\[ \det \rho_{E_1} =\det \rho_{E_2}=\chi,\]
where $\chi$ is the adelic cyclotomic character. 
It follows that
\[
G_{E_1 \times E_2} \subseteq \Delta(\widehat{\Z})
\quad \text{and} \quad 
G_{E_1 \times E_2}(n) \subseteq \Delta(\Z/n\Z),
\]
where, for any commutative ring $R$, we set
\[
\Delta(R) \coloneqq \{(M_1, M_2) \in \GL_2(R) \times \GL_2(R) : \det M_1 = \det M_2 \}.
\]
Accordingly, we regard the codomains of $\rho_{E_1 \times E_2}$ and $\rho_{E_1 \times E_2, n}$ to be $\Delta(\widehat{\Z})$ and $\Delta(\Z/n\Z)$, respectively.

Serre's open image theorem generalizes to the product setting as follows.

\begin{theorem}[Serre, \protect{\cite[Th\'eor\`eme 6]{MR0387283}}] \label{SOIT_Product}
Let \( E_1/\Q \) and \(E_2 / \Q\) be non-CM elliptic curves that are not $\overline{\Q}$-isogenous. Then \( G_{E_1 \times E_2} \) is an open subgroup of \( \Delta(\widehat{\Z}) \), and consequently the adelic index $[\Delta(\widehat{\Z}) : G_{E_1 \times E_2}]$ is finite.
\end{theorem}

We always have the containment
\begin{equation} \label{E:DetFiber}
    G_{E_1 \times E_2} \subseteq G_{E_1} \times_{\det} G_{E_2} \coloneqq (G_{E_1} \times G_{E_2}) \cap \Delta(\widehat{\Z})
\end{equation}
and similarly modulo $m$, for every positive integer $m$. This motivates the following definition.

\begin{definition}\label{entwinement}
Let $E_1$ and $E_2$ be non-CM elliptic curves that are not $\overline{\Q}$-isogenous. For a positive integer $m$, if
\[
G_{E_1 \times E_2}(m) \subsetneq G_{E_1}(m) \times_{\det} G_{E_2}(m),
\]
then we say that there is \emph{entwinement} between $E_1$ and $E_2$ at level $m$. The \emph{adelic entwinement factor} and the \emph{entwinement factor at level $m$} are defined by
\[
\delta_{E_1 \times E_2} \coloneqq 
[G_{E_1} \times_{\det} G_{E_2} : G_{E_1 \times E_2}] \quad \text{ and } \quad
\delta_{E_1 \times E_2}(m) \coloneqq 
[G_{E_1}(m) \times_{\det} G_{E_2}(m) : G_{E_1 \times E_2}(m)].
\]
The adelic entwinement factor is finite by \Cref{SOIT_Product}.
\end{definition}

Since $G_{E_1 \times E_2}$ is open in $\Delta(\widehat{\Z})$ under the hypotheses of \Cref{SOIT_Product}, there must exist a positive integer $m$ such that 
\[
G_{E_1 \times E_2} = \pi_m^{-1}\left(G_{E_1 \times E_2}(m)\right),
\]
where $\pi_m \colon \Delta(\widehat{\Z}) \to \Delta(\Z/m\Z)$ denotes the reduction modulo $m$ map.  The smallest such \( m \)
is called the \emph{adelic level} of the product \( E_1 \times E_2 \) and is denoted by \( m_{E_1 \times E_2} \). Note that \( m_{E_1 \times E_2} \) is necessarily a multiple of $\lcm(m_{E_1}, m_{E_2})$. Indeed, from
\[
G_{E_1 \times E_2} = \pi_{m_{E_1 \times E_2}}^{-1}\left(G_{E_1 \times E_2}(m_{E_1 \times E_2})\right),
\]
by projecting onto the first factor we see that $m_{E_1}$ divides $m_{E_1 \times E_2}$. Similarly $m_{E_2}$ divides $m_{E_1 \times E_2}$, and hence $\lcm(m_{E_1}, m_{E_2})$ divides $m_{E_1 \times E_2}$.

We conclude with a decomposition result for the mod $n$ image based on the adelic level.

\begin{proposition}\label{prop:Galois_image}
    Let $E_1$ and $E_2$ be elliptic curves satisfying the hypotheses of \Cref{SOIT_Product}. For any positive integers $d_1,d_2$ with $d_1\mid m_{E_1\times E_2}^\infty$ and $(d_2, m_{E_1\times E_2})=1$, we have 
    \[
    G_{E_1\times E_2}(d_1d_2) \simeq G_{E_1\times E_2}(d_1)\times \Delta(\Z/d_2\Z)
    \]
    via the isomorphism $\Delta(\Z/d_1d_2\Z) \to \Delta(\Z/d_1\Z) \times \Delta(\Z/d_2\Z)$ of the Chinese remainder theorem.
\end{proposition}
\begin{proof}
    This follows from the argument of \cite[Lemma 2.2]{Lee_Mayle_Wang_2025} and \cite[Lemma 8]{MR2439422}, mutatis mutandis. 
\end{proof}

\subsection{Serre Pairs}\label{S:SerrePairs}

Following Jones \cite{MR3071819}, a pair of elliptic curves $(E_1, E_2)$ over $\Q$ is a \emph{Serre pair} if
\[
[\Delta(\widehat{\Z}) : G_{E_1 \times E_2}] = 4.
\]
By an abuse of notation, we sometimes refer to the product $E_1 \times E_2$ as a Serre pair. By index considerations, a necessary condition for $(E_1, E_2)$ to be a Serre pair is that both $E_1$ and $E_2$ are Serre curves. This condition is not sufficient in general; see \Cref{S:SerrePairsII} for explicit criteria.

We now give an explicit description of the adelic image for Serre pairs.

\begin{proposition}\label{SerrePairGalImage}
    Let $E_1 \times E_2$ be a Serre pair. Then,
    \[
    G_{E_1 \times E_2} = G_{E_1} \times_{\det} G_{E_2}.
    \]
    Thus, for every positive integer $d$, we have
    $$G_{E_1\times E_2}(d) = G_{E_1}(d) \times_{\det} G_{E_2}(d).$$
\end{proposition}
\begin{proof}
    As noted above, $E_1$ and $E_2$ must both be Serre curves. By \eqref{E:fulldet}, we have 
    \[ \det G_{E_1} = \det G_{E_2} = \widehat{\Z}^\times.
    \]
    Hence, we may apply \Cref{lem:index_lemma} to the determinant fiber product to obtain
    $$[\Delta(\widehat{\Z}):G_{E_1}\times_{\det} G_{E_2}] = [\GL_2(\widehat{\Z}) : G_{E_1}] [\GL_2(\widehat{\Z}):G_{E_2}] = 2\cdot 2=4,$$
    where the indices may be computed by reducing to a sufficiently high finite level. By \eqref{E:DetFiber}, we have 
    \[ G_{E_1 \times E_2} \subseteq G_{E_1} \times_{\det} G_{E_2}.\] 
    Since both $G_{E_1}\times_{\det} G_{E_2}$ and $G_{E_1 \times E_2}$ have index $4$ in $\Delta(\widehat{\Z})$, they must coincide, proving the first claim. Reducing modulo $d$ proves the second claim.
\end{proof}

Let $E_1\times E_2$ be a Serre pair. Each $E_i$ is a Serre curve, and we denote its adelic level by $m_{E_i}$. By \Cref{SerrePairGalImage}, the adelic level of the product is
\begin{equation} \label{SerreLCM} m_{E_1 \times E_2} = \lcm(m_{E_1}, m_{E_2}). \end{equation}
Further, by Propositions \ref{SerreCurveGroup} and  \ref{SerrePairGalImage}, we have
\begin{equation}\label{SizeofSerrePair}
|G_{E_1\times E_2}(d)| =
\begin{cases}
|\Delta(\Z/d\Z)| 
& \text{if neither } m_{E_1} \text{ nor } m_{E_2} \text{ divides } d, \\[2pt]

\frac{1}{2} |\Delta(\Z/d\Z)| 
& \text{if exactly one of } m_{E_1} \text{ and } m_{E_2} \text{ divides } d, \\[2pt]

\frac{1}{4} |\Delta(\Z/d\Z)| 
& \text{if both } m_{E_1} \text{ and } m_{E_2} \text{ divide } d.
\end{cases}
\end{equation}

\begin{lemma}\label{4division}
    If $E_1 \times E_2$ is a Serre pair, then $\Q(E_1[4]) \cap \Q(E_2[4]) = \Q(i)$.
\end{lemma}
\begin{proof}
    See \cite[pp.~218--219]{MR3557121}. See also \Cref{criterion}.
\end{proof}

\begin{lemma}\label{m1andm2aredifferent}
    Let $E_1 \times E_2$ be a Serre pair. Then $m_{E_1} \neq m_{E_2}$.
\end{lemma}

\begin{proof}
We argue by contradiction. Suppose $m_{E_1} = m_{E_2}$. Since $E_1$ and $E_2$ are Serre curves, by \Cref{adeliclevelofSerrecurve}, we have $\Delta'_1 = \pm \Delta'_2$, where $\Delta'_i$ denotes the squarefree part of the discriminant of $E_i$. As shown in the proof of \Cref{Atleast6}, $\Delta'_i \neq \pm 1$. Then we have
$$\Q(i) \neq \Q(\sqrt{\Delta'_1},i) \subseteq \Q(E_1[4]) \cap \Q(E_2[4]),$$
which contradicts \Cref{4division}.
\end{proof}

Similarly to the single elliptic curve case, it is known that almost all pairs of elliptic curves are Serre pairs \cite{MR3071819}. We now recall the precise statement. For an elliptic curve $E/\mathbb{Q}$, write a short Weierstrass model as
$$E_{a, b} \colon y^2=x^3+ax+b,$$
where $a,b \in \mathbb{Z}$ and $\gcd(a^3,b^2)$ is $12$-th power free.
The naive height of $E$ is then defined by 
\[
H(E)=\max\{|a|^3, |b|^2\}. 
\]
For any real number $T \geq 1$, define
\begin{equation} \label{E:ET}
   \mathcal{F}(T) \coloneqq \{E/\Q : H(E) \leq T^6\} \quad \text{ and } \quad  \mathcal{E}(T) \coloneqq \{ (E_1/\Q, E_2/\Q) : \max( H(E_1), H(E_2)) \leq T^6 \}.
\end{equation}

As noted in \cite[p.~3383]{MR3071819}, we have $|\mathcal{F}(T)| \asymp T^5$ and $|\mathcal{E}(T)| \asymp T^{10}$. Define the subset
\[
\mathcal{E}_{\mathrm{non-Serre}}(T) \coloneqq \{ (E_1, E_2) \in \mathcal{E}(T) : E_1 \times E_2\text{ is not a Serre pair} \}.
\]
We have the following upper bound.
\begin{theorem}[Jones, \protect{\cite[Theorem 1.2]{MR3071819}}] \label{T:FewNonSerre} There is an explicit positive constant $\beta$ such that for any $T \geq 2$,
\[
|\mathcal{E}_{\mathrm{non-Serre}}(T)| \ll T^9 (\log T)^\beta
\]
with an absolute implied constant. 
\end{theorem}

It follows that as $T \to \infty$, the proportion of Serre pairs in $\mathcal{E}(T)$ tends to $1$. As we shall see in \Cref{S:AvgConst}, this means that in order to understand the average value of the coprimality constant, it will suffice to restrict our attention to Serre pairs.

\section{A Heuristic Model} \label{S:Model}

In this section, we define the constant $\Cco$ appearing in \Cref{C:Coprimality}. The guiding idea is to detect the condition $\gcd(\#E_1(\F_p), \#E_2(\F_p)) = 1$ by inclusion-exclusion, and to translate the resulting divisibility conditions into conditions on Frobenius conjugacy classes, to which the Chebotarev density theorem applies. Once $\Cco$ is defined, we establish in \Cref{AlmostEulerProduct} an ``almost Euler product'' expansion in which all local factors away from the adelic level are independent of the curves and arise from counting matrices in $\Delta(\mathbb{Z}/\ell\mathbb{Z})$.

Throughout this section, we let $E_1$ and $E_2$ be non-CM elliptic curves over $\Q$. For a positive integer $d$ and real number $x \geq 1$, define the counting function
\[
\mathcal{A}_d(x) \coloneqq
\# \{ p\le x: p\nmid N_{E_1}N_{E_2} \text{ and } d\mid \gcd(\#E_1(\F_p),\#E_2(\F_p)) \}. 
\]
We now express $\pico(x)$, defined in \eqref{E:pico}, as a sum involving $\mathcal{A}_d(x)$.

\begin{lemma}\label{lem:inclusion-exclusion}
    For all $x \geq 1$, we have that
    \[
    \pico(x) = \sum_{d \geq 1} \mu(d) \mathcal{A}_d(x).
    \]
    Moreover, the sum is finite since $\mathcal{A}_d(x) = 0$ for all $d > x + 1 + 2 \sqrt{x}$.
\end{lemma}
\begin{proof}
    For each prime $p \leq x$ with $p \nmid N_{E_1} N_{E_2}$, set 
    \[ g_p \coloneqq \gcd(\#E_1(\F_p), \#E_2(\F_p)).\]
    By a standard property of the M\"obius function \cite[Theorem 2.1]{MR434929},
    \[
    \sum_{d \mid g_p} \mu(d) = \begin{cases}
        1 & g_p = 1, \\
        0 & g_p > 1.
    \end{cases}
    \]
    Summing over primes $p \leq x$ of good reduction for both curves gives
    \[
    \pico(x) = \sum_{\substack{p \leq x \\ p \nmid N_{E_1}N_{E_2}}} \sum_{d \mid g_p} \mu(d).
    \]
    Since the double sum is finite, we may interchange the order of summation,
    \[ \pico(x) = \sum_{d \geq 1} \mu(d) \sum_{\substack{p \leq x \\ p \nmid N_{E_1}N_{E_2} \\ d \mid g_p}} 1 = \sum_{d \geq 1} \mu(d) \mathcal{A}_d(x). \]

    It remains to prove the finiteness claim. For each $i = 1,2$, the Hasse bound gives 
    \[ \# E_i(\F_p) \leq p + 1 + 2 \sqrt{p}.\]
    Hence for $p \leq x$, we have
    \[
    g_p \leq \min(\#E_1(\F_p),\#E_2(\F_p)) \leq p + 1 + 2 \sqrt{p} \leq x + 1 + 2 \sqrt{x}.
    \]
    If $d> x+1+2\sqrt{x}$, then $d\nmid g_p$ for every such $p$, and therefore $\mathcal{A}_d(x)=0$.
\end{proof}

We now reinterpret the divisibility condition defining $\mathcal{A}_d(x)$ in terms of Frobenius elements. First note that 
\[
d\mid \gcd(\#E_1(\F_p),\#E_2(\F_p)) \quad \iff \quad
d \mid \#E_1(\F_p) \text{ and } d \mid \#E_2(\F_p).
\]
For a prime $p \nmid d N_{E_i}$, we have by \eqref{E:TrDet} that
\[
\det(I - \rho_{E_i, d}(\Frob_p)) \equiv 1 - a_p(E_i) + p \equiv \#E_i(\F_p) \pmod{d}.
\]
In particular,
\[
d \mid \#E_i(\F_p)
\quad \iff \quad
\det(I - \rho_{E_i, d}(\Frob_p)) \equiv 0 \pmod{d}.
\]

Since $\mu(d) = 0$ unless $d$ is squarefree, we restrict our attention to squarefree $d$. Define 
\begin{equation} \label{E:Bd}
    \mathcal{B}_d \coloneqq \{(M_1, M_2) \in \Delta(\Z/d\Z) : \det(I - M_1) \equiv \det(I - M_2) \equiv 0 \pmod{d} \}.
\end{equation}
Then for any prime $p \nmid d N_{E_1} N_{E_2}$, we have
\[
d\mid \gcd(\#E_1(\F_p),\#E_2(\F_p)) \quad \iff \quad
(\rho_{E_1, d}(\Frob_p), \rho_{E_2, d}(\Frob_p)) \in \mathcal{B}_d.
\]
Since $G_{E_1 \times E_2}(d) \cap \mathcal{B}_d$ is stable under conjugation in $G_{E_1 \times E_2}(d)$, applying the Chebotarev density theorem with the field $\Q((E_1 \times E_2)[d]) = \Q(E_1[d], E_2[d])$ gives that
\begin{equation}\label{eq:chebotarev}
\mathcal{A}_d(x) \sim f_{E_1, E_2}(d) \cdot \frac{x}{\log x}
\end{equation}
as $x \to \infty$, where
\begin{equation} \label{E:fDef}
 f_{E_1, E_2}(d) \coloneqq \frac{\#(\mathcal{B}_d \cap G_{E_1\times E_2}(d))}{\#G_{E_1\times E_2}(d)}.
\end{equation}

Motivated by \Cref{lem:inclusion-exclusion} and the above discussion, we define the \emph{coprimality constant}
\begin{equation}
\label{E:Ccoprime}
    \Cco \coloneqq \sum_{d=1}^\infty \mu(d) f_{E_1, E_2}(d),
\end{equation}
whose convergence will be established in \Cref{AlmostEulerProduct}. One would like to substitute the approximation \eqref{eq:chebotarev} into \Cref{lem:inclusion-exclusion} to obtain an asymptotic for $\pico(x)$. However, even under the Generalized Riemann Hypothesis, the error terms of \eqref{eq:chebotarev} in the effective Chebotarev density theorem accumulate too rapidly when summed over $d$. This is the same obstacle that arises when attempting to prove the refined Koblitz conjecture, and at present we do not see a way to overcome it.

We now further assume that $E_1$ and $E_2$ are not isogenous over $\overline{\Q}$. To ease notation, we write $f \coloneqq f_{E_1, E_2}$ and $G \coloneqq G_{E_1 \times E_2}$. We now record a quasi-multiplicativity property of $f$ away from the finite set of primes dividing $m$.

\begin{lemma}\label{quasimultiplicative}
    For any positive integers $d_1$ and $d_2$ with $\gcd(d_1 m_{E_1\times E_2},d_2) = 1$, we have
    $$f(d_1d_2) = f(d_1) f(d_2).$$
\end{lemma}
\begin{proof}  By \Cref{prop:Galois_image} and the Chinese remainder theorem, the natural map
    $$\iota\colon G(d_1d_2) \to G(d_1) \times \Delta(\Z/d_2\Z)$$
    is an isomorphism. Moreover, $\iota$ restricts to a bijection 
\[
\mathcal{B}_{d_1d_2} \cap G(d_1d_2) \longrightarrow (\mathcal{B}_{d_1}\cap G(d_1)) \times (\mathcal{B}_{d_2}\cap \Delta(\Z/d_2\Z))
\]
since the hypothesis that $\gcd(d_1 m_{E_1\times E_2},d_2) = 1$ implies $\gcd(d_1, d_2) = 1$, and thus the congruences defining $\mathcal{B}_{d_1d_2}$ are equivalent to the simultaneous congruences defining $\mathcal{B}_{d_1}$ and $\mathcal{B}_{d_2}$. Therefore,
\begin{align*}
f(d_1d_2)
=\frac{\#(\mathcal{B}_{d_1d_2}\cap G(d_1d_2))}{\#G(d_1d_2)} 
=\frac{\#(\mathcal{B}_{d_1}\cap G(d_1))}{\#G(d_1)}
 \cdot
 \frac{\#(\mathcal{B}_{d_2}\cap \Delta(\Z/d_2\Z))}{\#\Delta(\Z/d_2\Z)}
= f(d_1)f(d_2),
\end{align*}
which completes the proof.
\end{proof}

We now deduce an ``almost Euler product'' expansion for $\Cco$ in which the dependence on $(E_1,E_2)$ is confined to primes dividing $m_{E_1\times E_2}$, and the local factors away from $m_{E_1\times E_2}$ are universal.

\begin{proposition}\label{AlmostEulerProduct}
    Let $E_1$ and $E_2$ be non-CM elliptic curves over $\Q$ that are not isogenous over $\overline{\Q}$. Then
    \begin{equation}\label{eulerproduct}
        \Cco = \left(\sum_{d \mid m_{E_1\times E_2}} \mu(d)f(d) \right) \prod_{\ell \nmid m_{E_1\times E_2}} \left(1 - \frac{(\ell+2)(\ell^2-\ell-1)}{(\ell-1)^3(\ell+1)^2}\right).
    \end{equation}
Moreover, the series in \eqref{E:Ccoprime} converges absolutely.
\end{proposition}

\begin{proof} Since $\mu(d) = 0$ unless $d$ is squarefree, we may write
\[
\Cco = \sum_{\substack{d \geq 1 \\ d \text{ squarefree}}} \mu(d) f(d).
\]
Every squarefree $d$ factors uniquely as $d = d_1 d_2$ with $d_1 \mid m_{E_1 \times E_2}$ and $(d_2, m_{E_1 \times E_2}) = 1$. For such a factorization, we have $\gcd(d_1 m_{E_1 \times E_2}, d_2) = 1$. 
Hence by \Cref{quasimultiplicative} and multiplicativity of $\mu$, we obtain
\[
\Cco = \left( \sum_{d_1 \mid m_{E_1 \times E_2}} \mu(d_1) f(d_1) \right) \left( \sum_{\substack{(d_2, m_{E_1 \times E_2}) = 1 \\ d_2 \text{ squarefree}}}\mu(d_2) f(d_2) \right).
\]
For squarefree $d_2$ with $\gcd(d_2, m_{E_1 \times E_2}) = 1$, we have $G(d_2) = \Delta(\Z/d_2\Z)$ by \Cref{prop:Galois_image}. It follows that
\[
\sum_{\substack{(d_2, m_{E_1 \times E_2}) = 1 \\ d_2 \text{ squarefree}}}\mu(d_2) f(d_2) = \prod_{\ell \nmid m_{E_1 \times E_2}} \left( 1 - f(\ell) \right),
\]
where for $\ell \nmid m_{E_1 \times E_2}$ we have $f(\ell) = \# \mathcal{B}_\ell / \# \Delta(\Z/\ell\Z)$. The proof now follows by applying the formula for $\#\mathcal{B}_\ell$ from \Cref{nonCMtimesnonCM} (which will be proved in the next section) and $\# \Delta(\Z/\ell\Z)$ from \eqref{SizeOfDelta}. Lastly, we note that since $f(\ell) \asymp \ell^{-2}$ for $\ell \nmid m_{E_1 \times E_2}$, the product over such $\ell$ converges absolutely, and therefore so does the series defining \eqref{E:Ccoprime}.
\end{proof}

\begin{remark}\label{eulerproductofnoncm}
    By \eqref{E:fDef} and the inclusion-exclusion principle, one sees that the finite series part in \eqref{eulerproduct} measures the proportion of elements $M$ in $G(m_{E_1\times E_2})$ for which $M \pmod d \not \in \mathcal{B}_d$ for each squarefree integer $d \mid m_{E_1\times E_2}$. Thus, since the infinite product is positive and less than 1, it follows that
    \[ 0 \leq \Cco < 1, \]
    provided that $E_1$ and $E_2$ satisfy the hypotheses of Proposition \ref{AlmostEulerProduct}.
\end{remark}

\section{Counting Matrices} \label{S:CountingMatrices}

In this section, we prove several matrix counting lemmas. We begin with the count for $\#\mathcal{B}_\ell$ that was used in the proof of \Cref{AlmostEulerProduct} for the ``universal'' local factors (those with $\ell \nmid m_{E_1 \times E_2}$). We then prove additional counts needed in \Cref{S:CcoSP}, where we specialize to the Serre pair setting.

Let $\ell$ be a prime. By viewing $\Delta(\Z/\ell\Z)$ as the fiber product of $\GL_2(\Z/\ell\Z)$ with itself over the
determinant map and using \eqref{eq:fiber_product}, we obtain
\begin{equation}\label{SizeOfDelta}
|\Delta(\Z/\ell\Z)| =\frac{|\GL_2(\Z/\ell\Z)|^2}{\ell-1} =\ell^2(\ell-1)^3(\ell+1)^2.
\end{equation}
If $n$ is squarefree, then by the Chinese remainder theorem,
\[
|\Delta(\Z/n\Z)| = \prod_{\ell \mid n} |\Delta(\Z/\ell\Z)|.
\]

Let $n$ be a squarefree positive integer and let $\alpha \in (\Z/n\Z)^\times$. Define
\begin{equation}\label{XnXnalpha}
    \begin{aligned}
           X_n &= \{ M \in \GL_2(\Z/n\Z) : \det(I-M) \equiv 0 \pmod n\}, \\
    X^\alpha_n &= \{M \in \GL_2(\Z/n\Z) : \det(I-M) \equiv 0 \pmod n \, \text{ and } \det(M) \equiv \alpha \pmod n\}.
    \end{aligned}
\end{equation}
Since $n$ is squarefree, the Chinese remainder theorem gives an isomorphism
\[ \GL_2(\Z/n\Z) \longrightarrow \prod_{ \ell \mid n} \GL_2(\Z/\ell\Z), \]
which restricts to bijections
\begin{equation}\label{correspondence}
    X_n \longrightarrow\prod_{ \ell \mid n} X_\ell \quad \text{ and } \quad X_n^\alpha \longrightarrow \prod_{ \ell \mid n} X_\ell^\alpha.
\end{equation}
Here (and elsewhere) we make an abuse of notation by writing $\alpha$ both for an element of $(\Z/n\Z)^\times$ and for its image in $(\Z/\ell\Z)^\times$.

\begin{lemma} \label{nonCMtimesnonCM} Let $\ell$ be a prime and $\mathcal{B}_\ell$ be as in \eqref{E:Bd}. Then
    $$|\mathcal{B}_{\ell}| =  \ell^2 (\ell + 2) (\ell^2-\ell-1).$$
\end{lemma}
\begin{proof} 
    For $\ell = 2$, the statement follows by a direct calculation, so assume $\ell$ is odd. Since the determinant map $\GL_2(\Z/\ell\Z) \to (\Z/\ell\Z)^\times$ is surjective, each fiber has size $|\SL_2(\Z/\ell\Z)| = \ell^3 - \ell$. 
    Together with \cite[Corollary 3.4]{Lee_Mayle_Wang_2025}, this gives the count
    \begin{equation} \label{E:Xla}
    |X_\ell^\alpha| = \begin{cases}
       \ell^2 &  \alpha \equiv 1 \pmod{\ell}, \\
        \ell^2+\ell & \alpha \not\equiv 1 \pmod{\ell}.
    \end{cases}
    \end{equation}
We may view $\mathcal{B}_\ell$ as the disjoint union of the sets $X_\ell^\alpha \times X_\ell^\alpha$ as $\alpha$ ranges over $(\Z/\ell\Z)^\times$. Thus
\[
    |\mathcal{B}_{\ell}|
    = |X_\ell^1|^2 + \sum_{\substack{\alpha \in (\Z/\ell \Z)^\times \\ \alpha \not \equiv 1 \pmod{\ell}}} |X_\ell^\alpha|^2 
    = \left( \ell^2\right)^2 + (\ell-2)  \left(\ell^2+\ell\right)^2.
\]
Expanding and factoring completes the proof.
\end{proof}

We now turn to the matrix counts needed in \Cref{S:CcoSP}.  In the Serre pair setting, the group $G_{E_1\times E_2}(d)$ is a subgroup of $\Delta(\Z/d\Z)$ of index at most $4$ with a constrained structure, which limits the number of  possibilities we need to consider. 

Let $n$ be an even squarefree integer, and let $\psi_n \colon \GL_2(\Z/n\Z) \to \{\pm 1\}$ be given by 
\[ \psi_n \coloneqq \prod_{\ell \mid n} \psi_\ell, \] where $\psi_\ell = (\frac{\det(\cdot)}{\ell})$ when $\ell$ is odd and $\psi_2 = \epsilon$ is the sign map. Here and throughout, we write $\left(\frac{\cdot}{\cdot}\right)$ to denote the Jacobi symbol. Further, we write $n_{\mathrm{odd}}$ for the odd part of $n$.

\begin{lemma} \label{l:formulaintpsix} With notation as above, for each $\alpha \in (\Z/n\Z)^\times$ we have
\[ 
|\psi_n^{-1}(+1) \cap X_n^\alpha| = \frac{1}{2}|X_n^\alpha| - \frac{1}{4}\left(\frac{\alpha}{n_{\mathrm{odd}}}\right)|X_n^\alpha|.
\]
\end{lemma}

\begin{proof}
     By \cite[Lemma 16]{MR2534114}, we have that
    \begin{equation} \label{E:L16}|\psi_n^{-1}(+1) \cap X_n^\alpha| = \frac{1}{2}\left(|X_n^\alpha| + \prod_{ \ell \mid n} \left(|\psi_\ell^{-1}(+1) \cap X_\ell^\alpha | - |\psi_\ell^{-1}(-1) \cap X_\ell^\alpha|\right)\right).\end{equation}
    Let $\ell$ be an odd prime dividing $n$. For all $M \in X_\ell^\alpha$, we have $\det M \equiv \alpha \pmod{\ell}$, so
    \[
    \psi_\ell(M) =  \left(\frac{\det M}{\ell}\right) =  \left(\frac{\alpha}{\ell}\right),
    \]
    and therefore,
    \begin{equation} \label{E:DiffOdd} 
    |\psi_\ell^{-1}(+1) \cap X_\ell^\alpha | - |\psi_\ell^{-1}(-1) \cap X_\ell^\alpha| = \left(\frac{\alpha}{\ell}\right) |X_\ell^\alpha|.
    \end{equation}
    For $\ell = 2$, we necessarily have $\alpha \equiv 1 \pmod{2}$, and $\psi_2 = \epsilon$. A direct check shows that
    \[ |\psi_2^{-1}(+1) \cap X_2^1| = 1 \quad \text{ and } \quad |\psi_2^{-1}(-1) \cap X_2^1| = 3 \]
    so
    \begin{equation} \label{E:DiffEven}  
    |\psi_2^{-1}(+1) \cap X_2^1 | - |\psi_2^{-1}(-1) \cap X_2^1| = -2 = - \frac{1}{2}  |X_2^1|. 
    \end{equation}
    Combining \eqref{correspondence}, \eqref{E:L16}, \eqref{E:DiffOdd}, and \eqref{E:DiffEven}, we establish the claim.
\end{proof}

Next, we prove an elementary lemma that will streamline later computations.

\begin{lemma}\label{multiplicativity}
For each prime power $\ell^\alpha$, let
\(
f_{\ell^\alpha} \colon (\Z/\ell^\alpha\Z)^\times \to \mathbb{C}
\)
be a function. For any positive integer $d$, we define $f_d : (\Z/d\Z)^\times \to \mathbb{C}$ by
\[
f_d(x) \coloneqq \prod_{\ell^\alpha\parallel d} f_{\ell^\alpha}(x \;\bmod \ell^\alpha).
\]
Set
\[
F(d)\coloneqq \sum_{x\in(\Z/d\Z)^\times} f_d(x).
\]
Then $F$ is multiplicative.
\end{lemma}

\begin{proof}
First, the function $f_d$ is well defined, since if $x\equiv x'\pmod d$, then 
\[
f_d(x)=\prod_{\ell^\alpha\parallel d} f_{\ell^\alpha}(x\;\bmod {\ell^\alpha})=\prod_{\ell^\alpha \parallel d} f_{\ell^\alpha}(x'\;\bmod \ell^\alpha)=f_d(x').
\]
To prove the lemma, it suffices to show that $F(\ell^\alpha e)=F(\ell^\alpha)F(e)$ for any prime power $\ell^\alpha$ and an integer $e$ coprime to $\ell$. Note that
\[
F(\ell^\alpha)F(e)
=\left(\sum_{x\in(\Z/\ell^\alpha\Z)^\times} f_{\ell^\alpha}(x)\right)
 \left(\sum_{y\in(\Z/e\Z)^\times} f_e(y)\right)
=\sum_{x\in(\Z/\ell^\alpha\Z)^\times}\sum_{y\in(\Z/e\Z)^\times} f_{\ell^\alpha}(x)f_e(y).
\]
By the Chinese remainder theorem, the pairs $(x,y) \in (\Z/\ell^\alpha\Z)^\times \times (\Z/e\Z)^\times$ are in one-to-one correspondence with $z\in(\Z/\ell^\alpha e\Z)^\times$, where
\[
z\equiv x \pmod {\ell^\alpha}
\qquad\text{and}\qquad
z\equiv y \pmod e,
\]
and hence we have $f_{\ell^{\alpha}}(x)f_e(y) = f_{\ell^{\alpha}e}(z)$.
Therefore,
\[
F(\ell^\alpha)F(e)
= \sum_{z\in(\Z/\ell^\alpha e\Z)^\times} f_{\ell^\alpha e}(z)
= F(\ell^\alpha e),
\]
completing the proof.
\end{proof}

Finally, let $n$ be  squarefree and let $r\mid n$. We introduce the sums
\begin{equation}\label{SnTnr}
    S(n) \coloneqq \sum_{\alpha \in (\Z/n\Z)^\times} |X_n^\alpha|^2 \quad \text{and} \quad T_r(n)\coloneqq \sum_{\alpha \in (\Z/n\Z)^\times}
\left(\frac{\alpha}{r_{\mathrm{odd}}}\right)|X_n^\alpha|^2,
\end{equation}
where $r_{\mathrm{odd}}$ is the odd part of $r$. Also define the multiplicative functions (on squarefree integers)
\begin{equation}
    F_1(d) \coloneqq \prod_{\ell \mid d} \frac{(\ell+2)(\ell^2-\ell-1)}{(\ell-1)^3(\ell+1)^2} \quad \text{ and } \quad F_2(d) \coloneqq \prod_{\ell \mid d} \frac{-(2\ell+1)}{(\ell-1)^3(\ell+1)^2}. \label{E:F1F2}
\end{equation}
Additionally, we define $F_1(1) = F_2(1) = 1$. Our last lemma for the section gives (relatively) simple formulas for the sums $S(n)$ and $T_r(n)$. 

\begin{lemma}\label{l:char_sum} 
    Let $n$ be a positive squarefree integer and $r \mid n$. Then
    \begin{align*}
        S(n) = |\Delta(\Z/n\Z)| F_1(n) \quad \text{ and } \quad T_r(n) = \begin{cases}
                    -\frac{4}{5}|\Delta(\Z/n\Z)| F_1(n/r) F_2(r) & r \text{ is even}, \\
        |\Delta(\Z/n\Z)| F_1(n/r) F_2(r) & r \text{ is odd}.
        \end{cases}
    \end{align*}
\end{lemma}

\begin{remark}
    At first glance, the formula for $T_r(n)$ appears to depend on the parity of $r$. However, in fact, if $2t \mid n$ and $t$ is odd and squarefree then $T_t(n)=T_{2t}(n)$ since
    \begin{align*}
        T_{2t}(n) &= -\frac{4}{5}|\Delta(\Z/n\Z)|F_1(n/2t)F_2(2t)\\
        &= -\frac{4}{5}|\Delta(\Z/n\Z)| \frac{F_1(n/t)}{F_1(2)}F_2(2)F_2(t)\\
        &= |\Delta(\Z/n\Z)| F_1(n/t)F_2(t) =T_t(n).
    \end{align*}
\end{remark}

\begin{proof}
    We first establish the formula for $S(n)$. For each $\ell \mid n$, define 
    \[ f_\ell(\alpha) \coloneqq |X_\ell^\alpha|^2\] for each $\alpha \in (\Z/\ell\Z)^\times$. By \eqref{correspondence} and \Cref{multiplicativity}, we have
    $$S(n) = \sum_{\alpha \in (\Z/n\Z)^\times} \prod_{\ell \mid n} |X_\ell^\alpha|^2 = \prod_{\ell \mid n} \sum_{\alpha \in (\Z/\ell\Z)^\times}|X_\ell^\alpha|^2.$$
    Using the formula for $|X_\ell^\alpha|$ from  \eqref{E:Xla}, we find that
    $$\sum_{\alpha \in (\Z/\ell\Z)^\times} |X_\ell^\alpha|^2 = (\ell^2)^2 + (\ell-2)(\ell^2+\ell)^2 = |\Delta(\Z/\ell\Z)|F_1(\ell).$$
    The formula for $S(n)$ now follows by the multiplicativity of $F_1$ on squarefree integers.

    We now prove the formula for $T_r(n)$. Here we define for each prime $\ell \mid n$,
    $$f_\ell^r(\alpha) \coloneqq \begin{cases}
    |X_2^1|^2 & \text{ if } \ell = 2,\\
    \left(\frac{\alpha}{\ell}\right) |X_\ell^\alpha|^2 & \text{ if } \ell\text{ is odd and } \ell \mid r, \\
    |X_\ell^\alpha|^2 & \text{ if } \ell \text{ is odd and } \ell \nmid r.
\end{cases}$$
By \Cref{multiplicativity},
\begin{equation} \label{E:Tnr1}
     T_r(n) = \sum_{\alpha \in (\Z/n\Z)^\times}\prod_{\ell \mid n} f_\ell^r(\alpha) = \prod_{\ell \mid n} \sum_{\alpha \in (\Z/\ell\Z)^\times} f_\ell^r(\alpha).
\end{equation}
First note that
\begin{equation} \label{E:Tnr2}
    |X_2^1|^2 = 16 = -\frac{4}{5}|\Delta(\Z/2\Z)| F_2(2).
\end{equation}
Next, if $\ell$ is odd and $\ell \mid r$, then
\begin{equation} \label{E:Tnr3}
    \sum_{\alpha \in (\Z/\ell\Z)^\times} \left(\frac{\alpha}{\ell}\right)|X_\ell^\alpha|^2 = (\ell^2)^2 + \left(\frac{\ell-1}{2}-1\right)(\ell^2+\ell)^2 - \left(\frac{\ell-1}{2}\right)(\ell^2+\ell)^2 = |\Delta(\Z/\ell\Z)|F_2(\ell).
\end{equation}
Lastly, if $\ell$ is odd and $\ell \nmid r$, then
\begin{equation} \label{E:Tnr4}
    \sum_{\alpha \in (\Z/\ell\Z)^\times}|X_\ell^\alpha|^2 = S(\ell) = |\Delta(\Z/\ell\Z)|F_1(\ell).
\end{equation}
Combining \eqref{E:Tnr1}, \eqref{E:Tnr2}, \eqref{E:Tnr3}, and \eqref{E:Tnr4}, and using the multiplicativity of $F_1$ and $F_2$ on squarefree integers, we obtain the desired formula for $T_r(n)$.
\end{proof}

\section{Coprimality Constant for Serre Pairs} \label{S:CcoSP}

In this section, we prove an explicit formula for $\Cco$ and bounds on it in the case that $E_1 \times E_2$ is a Serre pair. From \Cref{AlmostEulerProduct}, we see that it suffices to compute $f(d)$ only for squarefree divisors $d$ of $m_{E_1 \times E_2}$. We begin by proving the following preliminary lemma, which is related to \Cref{quasimultiplicative} and will be used in the proof of \Cref{calculatingfd}. 
Throughout this section, we use the notation
\[
m_1 \coloneqq m_{E_1}, \quad m_2 \coloneqq m_{E_2}, \quad m \coloneqq \gcd(m_1, m_2), \quad m' \coloneqq \frac{\lcm(m_1, m_2)}{\gcd(m_1, m_2)}.
\]

\begin{lemma}\label{Serremultiplicative}
    Let $E_1 \times E_2$ be a Serre pair. Let $d$ be a squarefree divisor of $m_{E_1 \times E_2}$. Assume that $m_{1} \mid d$ and $m_{2} \nmid d$. Let $d'$ be such that $d = d'm_{1}$.  Then $f(d) = f(d')f(m_{1})$.
\end{lemma}

\begin{proof}
    Since $d$ is squarefree and $m_1 \mid d$, the integer $m_1$ is squarefree and hence $\gcd(d', m_1) = 1$. Because $E_1 \times E_2$ is a Serre pair, \Cref{SerrePairGalImage} gives that 
    \begin{align}\label{longexpression}
        G_{E_1 \times E_2}(d) = G_{E_1}(d) \times_{\det}G_{E_2}(d).
    \end{align}
    By \cite[Lemma 2.2]{Lee_Mayle_Wang_2025},  \Cref{SerreCurveGroup}, and the assumption that $\gcd(d',m_1) = 1$, the right-hand side of \eqref{longexpression} is isomorphic to 
\begin{align*}
    (\GL_2(\Z/d'\Z) \times_{\det} \GL_2(\Z/d'\Z)) \times (G_{E_1}(m_1) \times_{\det} \GL_2(\Z/m_1\Z))
\end{align*}
    by the Chinese remainder theorem. Hence,
    $$G_{E_1 \times E_2}(d) \simeq \Delta(\Z/d'\Z) \times G_{E_1 \times E_2}(m_1).$$
    This isomorphism restricts to a bijection
    $$\mathcal{B}_{d} \cap G_{E_1 \times E_2}(d) \longrightarrow (\mathcal{B}_{d'} \cap \Delta(\Z/d'\Z)) \times (\mathcal{B}_{m_1} \cap G_{E_1\times E_2}(m_1))$$
    by the same argument as in the proof of \Cref{quasimultiplicative}. Therefore, we have
    \begin{align*}
        f(d'm_1) &= \frac{|\mathcal{B}_{d'm_1}\cap G_{E_1 \times E_2}(d'm_1)|}{|G_{E_1 \times E_2}(d'm_1)|} \\
        &= \frac{|\mathcal{B}_{d'} \cap \Delta(\Z/d'\Z)|}{|\Delta(\Z/d'\Z)|} \cdot \frac{|\mathcal{B}_{m_1} \cap G_{E_1 \times E_2}(m_1)|}{|G_{E_1\times E_2}(m_1)|} \\
        &= f(d')f(m_1),
    \end{align*}
    which completes the proof of the lemma.
\end{proof}

\begin{proposition} \label{calculatingfd} Let $E_1 \times E_2$ be a Serre pair and let $d$ be a positive squarefree divisor of $m_{E_1 \times E_2}$.
\begin{enumerate}
    \item If  $m_{1} \nmid d$ and $m_{2} \nmid d$, then
    \[
    f(d) = F_1(d).
    \]
    \item  If $m_{1} \mid d$ and $m_{2} \nmid d$, then
    \[
    f(d) = \left(1 + \frac{2}{5}  \frac{F_2(m_{1})}{F_1(m_{1})} \right) F_1(d).
    \]
    \item If $m_{1} \nmid d$ and $m_{2} \mid d$, then
    \[  f(d) = \left(1 + \frac{2}{5}  \frac{F_2(m_{2})}{F_1(m_{2})} \right) F_1(d).\]
    \item If $m_{1} \mid d$ and $m_{2} \mid d$, then $d = m_{E_1 \times E_2}$ and we have
\[
f(d) = \left( 1 + \frac{2}{5}  \frac{F_2(m_{1})}{F_1(m_{1})} + \frac{2}{5}  \frac{F_2(m_{2})}{F_1(m_{2})} + \frac{1}{4}  \frac{F_2(m')}{F_1(m')} \right) F_1(d) 
.\]
\end{enumerate}
\end{proposition}

\noindent \emph{Proof of Part (1).} Assume that $m_1 \nmid d$ and $m_2 \nmid d$. By \eqref{SizeofSerrePair}, we have that
$$G_{E_1\times E_2}(d) = \Delta(\Z/d\Z).$$
It follows that
\[
G_{E_1 \times E_2}(d) \cap \mathcal{B}_{d} = \{(M_1,M_2) \in \Delta(\Z/d\Z) : \det(I-M_1) \equiv \det(I-M_2) \equiv 0 \pmod {d}\}.
\]
Counting based on the value of the determinant $\det M_1 = \det M_2 \in (\Z/d\Z)^\times$, we see that
$$|G_{E_1\times E_2}(d) \cap \mathcal{B}_d| = \sum_{\alpha \in (\Z/d\Z)^\times}|X_d^\alpha|^2 = |\Delta(\Z/d\Z)| F_1(d),$$ 
by \Cref{l:char_sum} (see \eqref{XnXnalpha} and \eqref{SnTnr} for the relevant definitions). The first claim now follows since
\[
f(d) = \frac{|G_{E_1\times E_2}(d) \cap \mathcal{B}_d|}{|G_{E_1\times E_2}(d)|} = \frac{|\Delta(\Z/d\Z)| F_1(d)}{|\Delta(\Z/d\Z)|} = F_1(d).
\]

\smallskip

\noindent \emph{Proof of Parts (2) and (3).} Assume now that $m_1 \mid d$ and $m_2 \nmid d$. Since $d$ is squarefree, we note that $m_1$ is necessarily squarefree and we may write $d = m_1 d'$ with $\gcd(m_1, d') = 1$. Then $m_1 \nmid d'$ and $m_2 \nmid d'$. By Part (1) and \Cref{Serremultiplicative}, it then follows that
\[
f(d) = f(m_1) f(d') = f(m_1)F_1(d').
\]
Since $F_1$ is multiplicative, it therefore suffices to prove the claimed formula
in the case $d = m_1$.

By \eqref{SizeofSerrePair}, we have
\( |G_{E_1 \times E_2}(m_1)| = \frac{1}{2}|\Delta(\Z/m_1\Z)| \) and moreover
\[
G_{E_1\times E_2}(m_1)=G_{E_1}(m_1)\times_{\det}\GL_2(\Z/m_1\Z).
\]
It follows that
\[
G_{E_1 \times E_2}(m_1) \cap \mathcal{B}_{m_1} = \{(M_1,M_2) \in G_{E_1}(m_1) \times_{\det} \GL_2(\Z/m_1\Z) : \det(I-M_1) \equiv \det(I-M_2) \equiv 0 \pmod {m_1}\}.
\]
We again count based on the value of the determinant $\det M_1 = \det M_2 \in (\Z/m_1\Z)^\times$. Let $\psi_{m_1}$ be the quadratic character associated with $E_1$. Since $m_1$ is squarefree, $\psi_{m_1}$ has the local description discussed in \Cref{S:SerreCurves}. Using that $G_{E_1}(m_1) = \psi_{m_1}^{-1}(+1)$, we obtain
\[
    |G_{E_1 \times E_2}(m_1) \cap \mathcal{B}_{m_1}| = \sum_{\alpha \in (\Z/m_1\Z)^\times}|\psi^{-1}_{m_1}(+1) \cap X_{m_1}^\alpha| |X_{m_1}^\alpha|.
\]
By \Cref{l:formulaintpsix}, for each $\alpha \in (\Z/m_1\Z)^\times$ we have 
\begin{align*}
    |\psi_{m_1}^{-1}(+1) \cap X_{m_1}^\alpha| |X_{m_1}^\alpha| &= \frac{1}{2}|X_{m_1}^\alpha|
    ^2 - \frac{1}{4} \left(\frac{\alpha}{m_{1,\mathrm{odd}}}\right) |X_{m_1}^\alpha|^2.
\end{align*}
Summing over $\alpha$ gives
\begin{align*}
     |G_{E_1\times E_2}(m_1) \cap \mathcal{B}_{m_1}| = \frac{1}{2}\sum_{\alpha \in (\Z/m_1\Z)^\times}|X_{m_1}^\alpha|^2 - \frac{1}{4} \sum_{\alpha \in (\Z/m_1\Z)^\times}\left(\frac{\alpha}{m_{1,\mathrm{odd}}}\right) |X_{m_1}^\alpha|^2 = \frac{1}{2} S(m_1) - \frac{1}{4}T(m_1),
\end{align*}
where $S(m_1)$ is as in \eqref{SnTnr} and we write $T(m_1) \coloneqq T_{m_1}(m_1)$ for brevity.

Therefore, by \Cref{l:char_sum} and the fact that $m_1$ is even by \Cref{adeliclevelofSerrecurve},
\begin{align*}
    f(m_1) &= \frac{\frac{1}{2} S(m_1) - \frac{1}{4}T(m_1)}{\frac{1}{2}|\Delta(\Z/m_1\Z)|} \\
    &= \frac{\frac{1}{2}|\Delta(\Z/m_1\Z)| F_1(m_1) + \frac{1}{5}|\Delta(\Z/m_1\Z)| F_2(m_1)}{\frac{1}{2}|\Delta(\Z/m_1\Z)|} \\
    &= F_1(m_1) + \frac{2}{5} F_2(m_1),
\end{align*}
which completes the proof of Part (2). Part (3) follows by the same argument.

\smallskip

\noindent \emph{Proof of Part (4).} 
We now consider the final case. Because $d$ is squarefree, and $m_1\mid d$ and $m_2 \mid d$, by \eqref{SerreLCM}, we have $m_{E_1\times E_2}\mid d$. Hence $d=m_{E_1 \times E_2}$. Here we write $M \coloneqq m_{E_1 \times E_2}$ as shorthand. By our assumption that $d$ is squarefree, $M$ is necessarily squarefree. Our goal is to establish the formula for $f(M)$. By \eqref{SizeofSerrePair}, we have
\( |G_{E_1 \times E_2}(M)| = \frac{1}{4}|\Delta(\Z/M\Z)| \) and moreover
$$G_{E_1 \times E_2}(M) = G_{E_1} (M) \times_{\det} G_{E_2}(M).$$
Further, we have that
\[
G_{E_1 \times E_2}(M) \cap \mathcal{B}_{M} = \{(M_1,M_2) \in G_{E_1}(M) \times_{\det} G_{E_2}(M) : \det(I-M_1) \equiv \det(I-M_2) \equiv 0 \pmod {M}\}.
\]
We again count based on the value of the determinant $\det M_1 = \det M_2 \in (\Z/M\Z)^\times$,
\begin{equation} \label{E:Pt3-1}
|G_{E_1\times E_2}(M) \cap \mathcal{B}_M| = \sum_{\alpha \in (\Z/M\Z)^\times} |G_{E_1}(M) \cap X_M^{\alpha}| |G_{E_2}(M) \cap X_M^\alpha|.
\end{equation}
Let $m_1'$ and $m_2'$ be such that $M = m_1 m_1' = m_2 m_2'$. From  \cite[Lemma 2.2]{Lee_Mayle_Wang_2025}, we see that
\[
|G_{E_i}(M)\cap X_M^\alpha|
=|G_{E_i}(m_i)\cap X_{m_i}^{\alpha}||X_{m_i'}^{\alpha}|, \quad i\in \{1, 2\}.
\]

Let $\psi_{m_1}$ and $\psi_{m_2}$ be the quadratic characters associated with $E_1$ and $E_2$. Since $m_1$ and $m_2$ are squarefree, these characters have the local description discussed in \Cref{S:SerreCurves}.
Since \(G_{E_i}(m_i)=\psi_{m_i}^{-1}(+1)\), \Cref{l:formulaintpsix} gives
\[
|G_{E_i}(m_i)\cap X_{m_i}^{\alpha}|
=\frac12|X_{m_i}^{\alpha}|-\frac{1}{4}\left(\frac{\alpha}{m_{i,\mathrm{odd}}}\right)|X_{m_i}^{\alpha}|.
\]
Thus, since \(|X_M^\alpha|=|X_{m_i}^{\alpha}||X_{m_i'}^{\alpha}|,\)
\begin{equation} \label{E:Pt3-2}
  |G_{E_i}(M)\cap X_M^\alpha|
=\left(\frac12-\frac{1}{4}\left(\frac{\alpha}{{m_{i,\mathrm{odd}}}}\right)\right)|X_M^\alpha|.  
\end{equation}
Substituting \eqref{E:Pt3-2} into \eqref{E:Pt3-1} and expanding,
\begin{align*} |G_{E_1\times E_2}(M) \cap \mathcal{B}_M| 
    &= \frac{1}{4} \sum_{\alpha \in (\Z/M\Z)^\times} |X_M^\alpha|^2 - \frac{1}{8} \sum_{\alpha \in (\Z/M\Z)^\times}\left(\frac{\alpha}{m_{1,\mathrm{odd}}}\right)|X_M^{\alpha}|^2  \\
    & \qquad - \frac{1}{8}\sum_{\alpha \in (\Z/M\Z)^\times}\left(\frac{\alpha}{m_{2,\mathrm{odd}}}\right) |X_M^\alpha|^2
    + \frac{1}{16}\sum_{\alpha \in (\Z/M\Z)^\times}\left(\frac{\alpha}{m_{1,\mathrm{odd}}}\right)\left(\frac{\alpha}{m_{2,\mathrm{odd}}}\right)|X_M^\alpha|^2
\end{align*}

By \Cref{l:char_sum}, 
\begin{align*}
   \sum_{\alpha \in (\Z/M\Z)^\times} |X_M^\alpha|^2 &= S(M) = |\Delta(\Z/M\Z)| F_1(M) \\
   \sum_{\alpha \in (\Z/M\Z)^\times}\left(\frac{\alpha}{m_{i,\mathrm{odd}}}\right)|X_M^{\alpha}|^2 &= T_{m_i}(M) = -\frac{4}{5}|\Delta(\Z/M\Z)| F_1(m_i') F_2(m_i).
\end{align*}

Recall that $m = \gcd(m_1, m_2)$ and $m'$ is such that $M = m m'$. Note that $m$ is even so $m'$ must be odd. We also observe that $m_{1, \mathrm{odd}}\cdot m_{2, \mathrm{odd}}=m/2\cdot M/2 =(m/2)^2\cdot  m'$. Then
\[
\sum_{\alpha \in (\Z/M\Z)^\times}\left(\frac{\alpha}{m_{1,\mathrm{odd}}}\right)\left(\frac{\alpha}{m_{2,\mathrm{odd}}}\right)|X_M^\alpha|^2
= \sum_{\alpha \in (\Z/M\Z)^\times}\left(\frac{\alpha}{m'}\right)|X_M^\alpha|^2 = T_{m'}(M) = |\Delta(\Z/M\Z)| F_1(m) F_2(m').
\]
Putting these all together,
\begin{align*}
    |G_{E_1\times E_2}(M) \cap \mathcal{B}_M| &= \frac{1}{4} |\Delta(\Z/M\Z)| F_1(M) + \frac{1}{10} |\Delta(\Z/M\Z)| F_1(m_1') F_2(m_1) \\ 
    &\qquad + \frac{1}{10} |\Delta(\Z/M\Z)| F_1(m_2') F_2(m_2) + \frac{1}{16}|\Delta(\Z/M\Z)| F_1(m) F_2(m').
\end{align*}
By the multiplicativity of $F_1$ on squarefree integers,
\[
F_1(M) = F_1(m_1)F_1(m_1') = F_1(m_2)F_1(m_2') = F_1(m)F_1(m').
\]
Factoring out $\frac{1}{4} |\Delta(\Z/M\Z)| F_1(M)$, we obtain
\[
|G_{E_1\times E_2}(M) \cap \mathcal{B}_M| = \frac{1}{4} |\Delta(\Z/M\Z)| F_1(M) \left( 1 + \frac{2}{5} \frac{F_2(m_1)}{F_1(m_1)} + \frac{2}{5} \frac{F_2(m_2)}{F_1(m_2)} + \frac{1}{4} \frac{F_2(m')}{F_1(m')}   \right).
\]
Therefore,
\[
f(M) = \frac{|G_{E_1\times E_2}(M) \cap \mathcal{B}_M|}{|G_{E_1\times E_2}(M)|} =  \left( 1 + \frac{2}{5} \frac{F_2(m_1)}{F_1(m_1)} + \frac{2}{5} \frac{F_2(m_2)}{F_1(m_2)} + \frac{1}{4} \frac{F_2(m')}{F_1(m')}   \right) F_1(M).
\]
This completes the proof of Part (4).

We define multiplicative functions $F_1^*$ and $F_2^*$ for prime numbers $\ell$ by
\begin{equation}\label{F_1starF_2star}
    F_1^*(\ell) \coloneqq \frac{F_1(\ell)}{1-F_1(\ell)} \quad \text{ and } \quad F_2^*(\ell) \coloneqq \frac{|F_2(\ell)|}{1-F_1(\ell)},
\end{equation}
which we extend multiplicatively to define $F_1^*(n)$ and $F_2^*(n)$ for all squarefree integers $n \geq 1$.

\begin{theorem}\label{SerrePairConstant}
    Let $E_1 \times E_2$ be a Serre pair. Let $m_1$ and $m_2$ be the adelic levels of the Serre curves $E_1$ and $E_2$ respectively, and let $m = \gcd(m_1,m_2)$ and $m' = m_{E_1\times E_2}/m$. Then we have
    \begin{align*}
        \Cco = \begin{cases} \CcoGeneric & \text{ if } 4 \mid m_1\text{ and } 4 \mid m_2, \\
        \CcoGeneric\left( 1 + \dfrac{2}{5} F_2^*(m_1)\right) & \text{ if } 4 \nmid m_1\text{ and } 4 \mid m_2, \\
        \CcoGeneric\left( 1 + \dfrac{2}{5} F_2^*(m_2)\right) & \text{ if } 4 \mid m_1 \text{ and } 4 \nmid m_2,\\
        \CcoGeneric\left(1 + \displaystyle  \frac{2}{5} F_2^*(m_1) + \frac{2}{5} F_2^*(m_2) + \frac{\mu(m)}{4}F_1^*(m)F_2^*(m')\right) &\text{ if } 4 \nmid m_1\text{ and } 4 \nmid m_2.
        \end{cases}
    \end{align*}
\end{theorem}

\begin{proof}
   For notational convenience, we set $M \coloneqq \lcm(m_1,m_2)$, which equals the adelic level of the Serre pair by \eqref{SerreLCM}. Then
    \begin{equation}\label{E:SPC-start}
        \Cco=\left(\sum_{d\mid M}\mu(d)f(d)\right)\prod_{\ell\nmid M}(1-F_1(\ell))
        =\frac{\CcoGeneric}{\prod_{\ell\mid M}(1-F_1(\ell))} \cdot \sum_{d\mid M}\mu(d)f(d),
    \end{equation}
    by \Cref{AlmostEulerProduct} and \eqref{E:Cgeneric}. Recall that by \Cref{adeliclevelofSerrecurve}, the odd parts of $m_1$ and $m_2$ are squarefree. Thus, $4 \mid m_i$ if and only if $m_i$ is non-squarefree.  We now analyze the sum $\sum_{d\mid M}\mu(d)f(d)$ in the various cases according to whether or not $4\mid m_{i}$.

    Suppose that $4 \mid m_1$ and $4 \mid m_2$. By definition of the M\"obius function, only terms with squarefree $d$ contribute to the sum. In this case, there is no squarefree $d$ for which $m_1 \mid d$ or $m_2 \mid d$. Hence every $d$ that  contributes to the sum satisfies $m_1 \nmid d$ and $m_2 \nmid d$. By Part (1) of \Cref{calculatingfd}, $f(d) = F_1(d)$. Thus,
    $$\sum_{d \mid M} \mu(d)f(d) = \sum_{\substack{d \mid M}} \mu(d)F_1(d) = \prod_{\ell \mid M} (1-F_1(\ell)).$$
    Substituting into \eqref{E:SPC-start} gives $\Cco = \CcoGeneric$.

    Now suppose that $4 \nmid m_1$ and $4 \mid m_2$. Then $m_1$ is squarefree and $m_2$ is not, so if $d$ is squarefree, then necessarily $m_2 \nmid d$. By Parts (1) and (2) of \Cref{calculatingfd}, we have

\begin{equation}\label{case2calculation}
    \begin{aligned}
                \sum_{d\mid M} \mu(d)f(d) &= \sum_{\substack{d \mid M \\ m_1 \nmid d}}\mu(d)F_1(d) + \sum_{\substack{d \mid M \\ m_1 \mid d}}\mu(d)\left(1+ \frac{2}{5} \frac{F_2(m_1)}{F_1(m_1)}\right)F_1(d) \\
             &=\sum_{d\mid M}\mu(d)F_1(d) + \frac{2}{5}\frac{F_2(m_1)}{F_1(m_1)} \sum_{\substack{d' \mid \frac{M}{m_1} \\ (d',m_1) = 1}}\mu(d'm_1)F_1(d'm_1) \\
             &= \prod_{\ell \mid M}(1-F_1(\ell)) + \frac{2}{5}F_2(m_1) \mu(m_1) \sum_{\substack{d' \mid \frac{M}{m_1} \\ (d',m_1) = 1}}\mu(d')F_1(d').
    \end{aligned}
\end{equation}

    Recall that $F_2$ is a multiplicative function and that $F_2(\ell) < 0$ for any prime $\ell$ by definition. (See \eqref{E:F1F2}.) Hence, we have
    \begin{equation}\label{easyobservation1}
        F_2(m_1)\mu(m_1) = \prod_{\ell\mid m_1} -F_2(\ell) = \prod_{\ell \mid m_1} |F_2(\ell)|.
    \end{equation}
   Let $m_2'$ denote the odd part of $m_2/m$, that is, the product of the odd primes that divide $m_2$ but not $m_1$. Note that the conditions $d'm_1 \mid M$ and $\gcd(d',m_1) = 1$ are equivalent to $d' \mid m_2'$. Therefore, the sum in the last equation can be expressed as
    \begin{equation}\label{easyobservation2}
        \sum_{d' \mid m_2'} \mu(d')F_1(d') = \prod_{\ell \mid m_2'} (1-F_1(\ell)).
    \end{equation}
    Using \eqref{easyobservation1} and \eqref{easyobservation2}, \eqref{case2calculation} can be simplified as 
    \begin{align*}
        \prod_{\ell\mid M} (1-F_1(\ell)) + \frac{2}{5}\prod_{\ell \mid m_1} |F_2(\ell)| \cdot \prod_{\ell \mid m_2'} (1-F_1(\ell)) &= \prod_{\ell \mid M} (1-F_1(\ell))\cdot  \left( 1 + \frac{2}{5} \prod_{\ell \mid m_1} \frac{|F_2(\ell)|}{1-F_1(\ell)}\right).
    \end{align*}
    This completes the proof of the second case. The third case can be derived using the same argument.

    Finally, we assume $4\nmid m_1$ and $4\nmid m_2$. In this case, $m_1$, $m_2$, and $M$ are all squarefree.  We begin by splitting the sum
\begin{align*}
    \sum_{d \mid M} \mu(d)f(d) &= 
    \sum_{\substack{d \mid M \\ m_1 \nmid d \\ m_2 \nmid d}}\mu(d)F_1(d)
    + \sum_{\substack{d \mid M \\ m_1 \mid d \\ m_2 \nmid d}}\mu(d)\left(1+\frac{2}{5} \frac{F_2(m_1)}{F_1(m_1)}\right) F_1(d) 
    + \sum_{\substack{d \mid M \\ m_1 \nmid d \\ m_2 \mid d}}\mu(d)\left(1+\frac{2}{5} \frac{F_2(m_2)}{F_1(m_2)}\right) F_1(d) \\
    &\qquad+ \mu(M) \left( 1 + \frac{2}{5} \frac{F_2(m_1)}{F_1(m_1)} + \frac{2}{5} \frac{F_2(m_2)}{F_1(m_2)} + \frac{1}{4} \frac{F_2(m')}{F_1(m')}\right) F_1(M).
\end{align*}
    Then
    \begin{align*}
        \sum_{d \mid M} \mu(d)f(d) &= \sum_{d\mid M} \mu(d)F_1(d) + \frac{2}{5}\frac{F_2(m_1)}{F_1(m_1)} \sum_{d' \mid \frac{M}{m_1}}\mu(d'm_1)F_1(d'm_1) \\
        &\qquad+ \frac{2}{5}\frac{F_2(m_2)}{F_1(m_2)} \sum_{d' \mid \frac{M}{m_2}}\mu(d'm_2)F_1(d'm_2) + \frac{1}{4} \frac{F_2(m')}{F_1(m')}\mu(M)F_1(M).
    \end{align*}
    Observe that
    \begin{align*}
        \sum_{d\mid M} \mu(d)F_1(d) &= \prod_{\ell \mid M} (1-F_1(\ell)) \\
        \frac{2}{5}\frac{F_2(m_1)}{F_1(m_1)} \sum_{d' \mid \frac{M}{m_1}}\mu(d'm_1)F_1(d'm_1) &= \frac{2}{5} \prod_{\ell \mid M} ( 1-F_1(\ell)) \prod_{\ell \mid m_1} \frac{|F_2(\ell)|}{1-F_1(\ell)} \\
        \frac{2}{5}\frac{F_2(m_2)}{F_1(m_2)} \sum_{d' \mid \frac{M}{m_2}}\mu(d'm_2)F_1(d'm_2) &= \frac{2}{5} \prod_{\ell \mid M} (1-F_1(\ell)) \prod_{\ell \mid m_2} \frac{|F_2(\ell)|}{1-F_1(\ell)} \\
        \frac{1}{4} \frac{F_2(m')}{F_1(m')}\mu(M)F_1(M) &= \frac{\mu(m)}{4}\prod_{\ell \mid M}(1-F_1(\ell)) \prod_{\ell \mid m}  \frac{F_1(\ell)}{1-F_1(\ell)}  \prod_{\ell \mid m'} \frac{|F_2(\ell)|}{1-F_1(\ell)}.
    \end{align*}
    Therefore, 
    \[ 
    \sum_{d \mid M} \mu(d)f(d) =\prod_{\ell\mid M}(1-F_1(\ell))\left(1+\frac{2}{5}F_2^*(m_1)+\frac{2}{5}F_2^*(m_2)+\frac{\mu(m)}{4} F_1^*(m) F_2^*(m')\right),
    \]
    which completes the proof of the final piece of the formula by \Cref{AlmostEulerProduct}.
\end{proof}

Now, we proceed to the proof of  \Cref{C:SerrePairBounds}. From \eqref{E:F1F2} and \eqref{F_1starF_2star}, we have that
\begin{equation}\label{F1starF2star}
        F^*_1(\ell)=\frac{\ell^3+\ell^2-3\ell-2}{\ell^5-\ell^4-3\ell^3+\ell^2+4\ell+1}
    \quad \text{and} \quad
    F_2^*(\ell)=\frac{2\ell+1}{\ell^5-\ell^4-3\ell^3+\ell^2+4\ell+1},
\end{equation}
and that they are defined multiplicatively on squarefree integers.
\begin{lemma}\label{easylemma}
For each $i \in \{1,2\}$, the following properties hold.
\begin{enumerate}
        \item For any distinct primes $p > q$ and any squarefree integer $t$ coprime to $p$, we have $F_i^*(pt) < F_i^*(q)$.
        \item For any odd squarefree integer $t\ge 3$, we have  $F_i^*(2t) \leq F_i^*(6)$.
    \end{enumerate}
\end{lemma}
\begin{proof}
    
    Using Calculus and the formulas appearing in \eqref{F1starF2star}, one can confirm that each $F_i^*$ is a monotonically decreasing function on the primes. Hence, $F_i^*(p)<F_i^*(q)$. Since $F_i^*$ is multiplicative and $0 < F_i^*(\ell) \leq 1$ for any prime $\ell$, $0 < F_i^*(t) \leq 1$ also holds for any squarefree integer $t$. Therefore, we get $F_i^*(pt)=F_i^*(p)F_i^*(t)<F_i^*(q)$. 
    
    From Part (1), we get  $F_i^*(t) \leq F_i^*(3)$ for any odd squarefree integer $t\geq 3$. Hence, the second claim follows from $F_i^*(2t)=F_i^*(2)F_i^*(t)\leq F_i^*(2)F_i^*(3)$.  
\end{proof}

\begin{proof}[Proof of \Cref{C:SerrePairBounds}]
Let $E_1 \times E_2$ be a Serre pair with adelic levels $m_1, m_2$. We define the ratio 
$$R(m_1,m_2) \coloneqq \frac{\Cco}{\CcoGeneric}.$$ 
Our goal is to prove that
\begin{equation} \label{E:RGoal}
R(70,210)\leq R(m_1,m_2) \leq R(6,10).
\end{equation}
This will establish the desired bound, since from \eqref{F1starF2star} one computes that
\[
R(70,210) = \frac{5014419112}{5014521525}
\quad \text{and} \quad
R(6,10) = \frac{1150648}{1118065}. \]
To establish these bounds, we proceed by cases. First note that if $4 \mid m_1$ and $4 \mid m_2$, then $R(m_1,m_2)=1$ by \Cref{SerrePairConstant}, so the bounds in \eqref{E:RGoal} hold. Next, suppose that $4 \nmid m_1$ and $4 \mid m_2$. Then \Cref{SerrePairConstant} gives 
\[
R(m_1, m_2) = 1 + \dfrac{2}{5} F_2^*(m_1) > 1 > R(70, 210),
\]
which establishes the lower bound. For the upper bound, \Cref{Atleast6} together with the assumption that $4 \nmid m_1$ implies that $m_1$ has an odd prime factor. Hence, by \Cref{easylemma},
$$R(m_1,m_2) = 1 + \frac{2}{5}F_2^*(m_1) \leq 1 + \frac{2}{5}F_2^*(6) < R(6,10),$$
establishing the upper bound in this case. The case $4 \mid m_1$ and $4 \nmid m_2$ is handled similarly. 

It remains to treat the case in which both $m_1$ and $m_2$ are squarefree.  By \Cref{m1andm2aredifferent}, we may assume without loss of generality that $m_1 < m_2$. We will show, case by case, that $R(m_1,m_2) < R(6,10)$ whenever $(m_1,m_2) \neq (6,10)$.

First assume that $m_1 = 6$ and $m_2 > 10$.
 If $3\mid m_2$, then $m_2 = 6t$ for some odd integer $t > 1$. In this case, $t$ must have an odd prime factor at least $5$. By \Cref{easylemma}, we obtain 
 \[
F_2^*(t) \leq F_2^*(5) \quad \text{and} \quad F_2^*(6t) = F_2^*(6)F_2^*(t) \leq F_2^*(6)F_2^*(5)=F_2^*(30),
 \]
 and hence
$$R(6,6t) \leq 1 + \frac{2}{5}F_2^*(6) + \frac{2}{5}F_2^*(30)+ \frac{1}{4}F_1^*(6) F_2^*(5) < R(6,10).$$
If $3 \nmid m_2$ and $m_2 > 10$, then $m_2$ must have a prime factor at least $7$, so we may write $m_2 = 2t$, where $t$ is divisible by such a prime. Similar to before, we have
$$R(6,2t) \leq 1 + \frac{2}{5}F_2^*(6) + \frac{2}{5}F_2^*(14) + \frac{1}{4}F_1^*(2)F_2^*(21) < R(6,10).$$
This establishes the upper bound when $m_1=6$.

Now suppose that $m_1 > 6$. Then both $m_1$ and $m_2$ are squarefree, and each must have an odd prime factor at least $5$. Thus, by \Cref{easylemma},
$$R(m_1,m_2) \leq 1 + \frac{2}{5}F_2^*(10) + \frac{2}{5}F_2^*(10) + \frac{1}{4}F_1^*(2) F_2^*(5) < R(6,10).$$
We conclude that the upper bound occurs only when $(m_1,m_2) = (6,10)$. \Cref{Ex:High} gives an example of a Serre pair with these adelic levels, so the upper bound is sharp.

We now prove the lower bound in the remaining case that $m_1$ and $m_2$ are both squarefree. By \Cref{SerrePairConstant}, in order for $R(m_1,m_2)<1$,  the integer $m$ must have an odd number of prime factors. Moreover, by \Cref{m1andm2aredifferent}, $m'$ must be an odd integer greater than $1$.

\emph{Case 1:} Suppose $m = 2$. Write $m_1 = a_1m$ and $m_2 = a_2m$ for some coprime odd integers $a_1$ and $a_2$. Since $F_2^*(2) = 1$ and $F_1^*(2) = \frac{4}{5}$, we obtain 
$$R(m_1,m_2) = 1 + \frac{2}{5}F_2^*(a_1) + \frac{2}{5}F_2^*(a_2) - \frac{1}{5}F_2^*(a_1)F_2^*(a_2).$$
By \Cref{easylemma}, we have $F_2^*(a_i) \leq F_2^*(3) = \frac{7}{103}$. Observe that for all real numbers $0 < x,y \leq \frac{7}{103}$,
\[
2x+2y-xy > 0.
\]
It follows that
\[
\frac{2}{5}F_2^*(a_1) + \frac{2}{5}F_2^*(a_2) - \frac{1}{5}F_2^*(a_1)F_2^*(a_2) > 0,
\]
and consequently $R(m_1,m_2) > 1 > R(70, 210)$.

\emph{Case 2:} Suppose $3\mid m$. Then there must exist an odd prime, at least $5$, that is also a divisor of $m$. Likewise, $m'$ must have a prime factor at least $5$. By \Cref{easylemma}, we have 
\[
F_1^*(m) \leq F_1^*(5)F_1^*(6)\quad  \text{  and } \quad  F_2^*(m') \leq F_2^*(5).
\]
Hence, a direct calculation shows that
\[
R(m_1,m_2) >1 - \frac{1}{4}F_1^*(30)F_2^*(5) > R(70,210).
\]

\emph{Case 3:} Suppose $3 \nmid m$. Then $m$ has at least two odd prime divisors, both at least $5$.

\emph{Case 3a:} Suppose that $m$ has an odd prime divisor that is at least $11$. Since $F_1^*(m)\le 1$ and by \Cref{easylemma}, 
we have 
\[
F_1^*(m) \leq F_1^*(11)F_1^*(5)F_1^*(2) \quad \text{ and }   \quad F_2^*(m') \leq F_2^*(3).
\]
A calculation shows that
\[
R(m_1, m_2) > 1 - \frac{1}{4}F_1^*(110)F_2^*(3)  > R(70,210).
\]

\emph{Case 3b:} Assume that all prime factors of $m$ are less than 11 and that $3 \nmid m'$. Then necessarily $m = 2\cdot 5\cdot 7=70$, and $F_2^*(m') \leq F_2^*(11)$. A calculation shows that
$$R(m_1,m_2) > 1 - \frac{1}{4}F_1^*(70)F_2^*(11) > R(70,210).$$

\emph{Case 3c:} Assume that all prime factors of $m$ are less than 11 and that $3 \mid m'$. Then $m=70$ and $F_2^*(m')\leq F_2^*(3)$. Write $m_1=70 a_1$,  $m_2=70 a_2$ for some coprime odd integers $a_1, a_2$. Moreover, we have $a_1\geq 1$ and $a_2\geq 3$.
Hence,
\[
R(m_1, m_2) = 1 +\frac{2}{5}F_2^*(70)F_2^*(a_1)+\frac{2}{5}F_2^*(70)F_2^*(a_2)- \frac{1}{4}F_1^*(70)F_2^*(a_1)F_2^*(a_2).
\]
Note that $F_2^*(a_1)\leq 1$, and by \Cref{easylemma}, we have  $F_2^*(a_2) \leq F_2^*(3) = \frac{7}{103}$. Observe that for all real numbers $0 < x\le 1, 0<y \leq \frac{7}{103}$,
\[
1 +\frac{2}{5}F_2^*(70)x+\frac{2}{5}F_2^*(70)y- \frac{1}{4}F_1^*(70)xy\geq 1 +\frac{2}{5}F_2^*(70)+\frac{2}{5}F_2^*(70)F_2^*(3)- \frac{1}{4}F_1^*(70)F_2^*(3).
\]
It follows that
 $R(m_1,m_2) \geq  R(70, 210)$ with equality when $m_1=70$ and $m_2=70\cdot 3=210.$ Thus the lower bound occurs only when $(m_1,m_2) = (70,210)$. \Cref{Ex:Low} gives an example of a Serre pair with these adelic levels, so the lower bound is sharp.
\end{proof}

\section{Average Constants and Moments} \label{S:AvgConst}

In this section, we prove \Cref{averageresult}. We begin by proving the following proposition. 

\begin{proposition}\label{asymptoticbound}
    Let $E_1 \times E_2$ be a Serre pair and $m_i$ be the adelic level of the curve $E_i$ for $i = 1, 2$. Then 
    $$|\Cco-\CcoGeneric| \ll \frac{1}{\rad(m_1)^3} + \frac{1}{\rad(m_2)^3} + \frac{1}{\min\{\rad(m_1),\rad(m_2)\}}.$$ Here the notation $\ll$ means the left hand side is less than an absolute constant (not depending on $E_1, E_2, m_1$, and $m_2$) times the right hand side.  
\end{proposition}

\begin{proof}
    By \eqref{F1starF2star}, we have that $F_1^*(p) \leq \frac{1}{p}$ for any primes $p \geq 3$ and $F_2^*(p) \leq \frac{1}{p^3}$ for all primes $p \geq 5$. Thus, for any squarefree positive integer $n$, we have
    $$F_1^*(n) \ll \frac{1}{n} \quad \text{ and } \quad F_2^*(n) \ll \frac{1}{n^3}.$$
    The desired result for the cases $4 \mid m_1$ or $4 \mid m_2$ now follows from \Cref{SerrePairConstant}.
    
    Now suppose that $m_1$ and $m_2$ are squarefree. Then $m$ and $m'$ (defined as in \Cref{S:CcoSP}) are squarefree and $mm' \geq \min\{m_1,m_2\}$.  Therefore, we have
    $$F_1^*(m)F_2^*(m') \ll \frac{1}{m}\cdot \frac{1}{m'^3} \leq \frac{1}{\min\{m_1,m_2\}}.$$
    This concludes the proof.
    \end{proof}

\Cref{asymptoticbound} remains valid even without the use of $\rad$, since $m_1$ and $m_2$ are squarefree except possibly at the prime $2$, whose exponent can be at most three. However, we retain the formulation with $\rad$ for convenience in the proof of \Cref{averageresult}, which follows below.

Let $T > 0$ and $\mathcal{E}(T)$ be as defined in \eqref{E:ET}. Let $\mathcal{E}_{\text{Serre}}(T)$ and $\mathcal{E}_{\text{non-Serre}}(T)$ be the subfamilies of Serre pairs and non-Serre pairs respectively. We aim to prove that
$$\frac{1}{|\mathcal{E}(T)|} \left(\sum_{(E_1,E_2) \in \mathcal{E}_{\text{Serre}}(T)} |C^{\text{coprime}}_{E_1,E_2} - C^{\text{coprime}}|^t + \sum_{(E_1,E_2)\in \mathcal{E}_{\text{non-Serre}}(T)} |C^{\text{coprime}}_{E_1,E_2} - C^{\text{coprime}}|^t\right) \to 0,$$
where if one of $E_1$ or $E_2$ has CM, or $E_1$ and $E_2$ are $\overline{\Q}$-isogenous, then we do not give a conjectural constant for the density of coprime reduction, but rather use the definition 
\begin{equation}\label{eq:CM}
  C_{E_1,E_2}^{\text{coprime}}\coloneqq\limsup_{x\to \infty} \frac{\pico(x)}{x/\log x}.  
\end{equation}
Then, together with \Cref{eulerproductofnoncm}, we have $|C_{E_1,E_2}^{\text{coprime}} - C^{\text{coprime}}| \leq 1$ for all pairs $(E_1, E_2)$ of elliptic curves. Therefore, by \Cref{T:FewNonSerre},
\begin{equation}\label{NonSerrePairBound}
    \frac{1}{|\mathcal{E}(T)|} \sum_{(E_1,E_2) \in \mathcal{E}_{\text{non-Serre}}(T)} |C^{\text{coprime}}_{E_1,E_2} - C^{\text{coprime}}|^t \leq \frac{|\mathcal{E}_{\text{non-Serre}}(T)|}{|\mathcal{E}(T)|}\ll \frac{(\log T)^\beta}{T}.
\end{equation}
Now, we focus on the Serre pair subfamily. Let $E_1 \times E_2$ be a Serre pair, and let us use the same notation as in \Cref{SerrePairConstant}. By \Cref{asymptoticbound}, we have
    $$|\Cco - \CcoGeneric| \ll \frac{1}{\rad(m_1)^3} + \frac{1}{\rad(m_2)^3} + \frac{1}{\min\{\rad(m_1),\rad(m_2)\}}.$$
We set $D_i$ to be the absolute value of the squarefree part of $4a_i^3+27b_i^2$, $1\leq i\leq 2$. By \Cref{adeliclevelofSerrecurve}, we have

\begin{align} \label{Serrepairbound}
    \sum_{\substack{(E_1,E_2) \in \mathcal{E}_\text{Serre}(T)}}|C^{\text{coprime}}_{E_1,E_2} - C^{\text{coprime}}|^t \leq \sum_{\substack{|a_1| \leq T^2 \\ |b_1|\leq T^3 \\ 4a_1^3+27b_1^2\neq 0}}
 \sum_{\substack{|a_2| \leq T^2 \\ |b_2| \leq T^3 \\ 4a_2^3+27b_2^2 \neq 0}} \left(\frac{1}{D_1^3} + \frac{1}{D_2^3} + \frac{1}{\min\{D_1,D_2\}} \right)^t.
\end{align}
For any real numbers $x,y,z \geq 0$ and a positive integer $t$, one can check that
$$(x+y+z)^t \leq 3^{t-1} (x^t+y^t+z^t),$$
and hence the right-hand side of \eqref{Serrepairbound} can be bounded by
\begin{align*}
\ll_t |\mathcal{F}(T) | \cdot  \sum_{\substack{|a_1| \leq T^2 \\ |b_1| \leq T^3 \\ 4a_1^3+27b_1^2 \neq 0}} \frac{1}{D_1^{3t}} + |\mathcal{F}(T) | \cdot \sum_{\substack{|a_2| \leq T^2 \\ |b_2| \leq T^3 \\ 4a_2^3+27b_2^2 \neq 0}} \frac{1}{D_2^{3t}} + \sum_{\substack{|a_1| \leq T^2 \\ |b_1|\leq T^3 \\ 4a_1^3+27b_1^2\neq 0}}
 \sum_{\substack{|a_2| \leq T^2 \\ |b_2| \leq T^3 \\ 4a_2^3+27b_2^2 \neq 0}} \frac{1}{(\min\{D_1,D_2\})^{t}},
 \end{align*}
 where $\mathcal{F}(T)$ is defined in \eqref{E:ET}. We use the following lemma.
 \begin{lemma} \label{JonesLemma}
     Let $A$ and $B$ be positive real numbers and $k$ be a positive integer. Then
     $$\sum_{\substack{|a| \leq A \\ |b| \leq B \\ 4a^3+27b^2 \neq 0}} \frac{1}{|(4a^3+27b^2)_{\text{sf}}|^k} \ll B + AB^\frac{1}{k+1}(\log B \cdot (\log A)^7)^\frac{k}{k+1},$$
     where $(4a^3+27b^2)_\text{sf}$ denotes the squarefree part of $4a^3+27b^2$.
 \end{lemma}
\begin{proof}
    See \cite[Chapter 4.2]{MR2534114}. In particular, this is equivalent to \cite[(22)]{MR2534114} after multiplying both sides by $AB$.
\end{proof}
Let us take $A = T^2$ and $B = T^3$ in \Cref{JonesLemma}. Then we have that for any $t\geq 1$, 
\begin{equation}\label{Jones1}
|\mathcal{F}(T)|\sum_{\substack{|a_1|\leq T^2 \\ |b_1|\leq T^3 \\ 4a_1^3+27b_1^2\neq 0}} \frac{1}{D_1^{3t}} \ll T^5 \left( T^3 + T^{2 + \frac{3}{3t+1}} (\log T)^{\frac{24t}{3t+1}}\right) = o_t(T^9)    .
\end{equation}
 On the other hand, observe that
$$\frac{1}{(\min\{D_1,D_2\})^t} \leq \frac{1}{D_1^t} + \frac{1}{D_2^t},$$
and hence, 
 $$\sum_{\substack{|a_1|\leq T^2 \\ |b_1| \leq T^3 \\ 4a_1^3+27b_1^2 \neq 0}} \sum_{\substack{|a_2| \leq T^2 \\ |b_2|\leq T^3 \\ 4a_2^3+27b_2^2 \neq 0}} \frac{1}{(\min\{D_1,D_2\})^t} \leq 2|\mathcal{F}(T)| \sum_{\substack{|a| \leq T^2 \\ |b| \leq T^3 \\ 4a^3+27b^2 \neq 0}} \frac{1}{|(4a^3+27b^2)_{\text{sf}}|^t}.$$
By \Cref{JonesLemma}, we have 
 \begin{equation}\label{Jones2}
     |\mathcal{F}(T)|\sum_{\substack{|a| \leq T^2 \\ |b| \leq T^3 \\ 4a^3+27b^2 \neq 0}} \frac{1}{|(4a^3+27b^2)_{\text{sf}}|^t} \ll T^5\left(T^3 + T^{2+\frac{3}{t+1}} (\log T)^{\frac{8t}{t+1}}\right) = o_t(T^{9}).
 \end{equation}
 Finally, since $|\mathcal{E}(T)| \asymp T^{10}$, combining \eqref{NonSerrePairBound}, \eqref{Serrepairbound}, \eqref{Jones1}, and \eqref{Jones2}, \Cref{averageresult} follows.

\section{Numerical Examples} \label{S:NumEx}

In this section, we give some numerical examples where there are only finitely many good primes of coprime order reduction, and some examples related to Serre pairs.

\subsection{An Example with Only Finitely Many Primes of Coprime Reduction}

If $E$ has a rational $\ell$-torsion for some prime $\ell$, then $\ell$ divides $\#E(\F_p)$ for any prime $p \nmid N_E$ (see \cite{MR604840}). Thus, it is clear that if $E_1$ and $E_2$ both have rational $\ell$-torsions, then $\pico(x)$ is absolutely bounded.

However, we can have more interesting examples where $\pico(x)$ is absolutely bounded, similar to the example of Jones appearing in \cite{MR2805578} pertaining to the Koblitz constant. Let $E_1$ be the elliptic curve \texttt{484.a1} and $E_2$ be \texttt{847.c1}, which are given by the models
\begin{align*}
    E_1 &\colon y^2 = x^3+x^2-9357x+347279, \\
    E_2 &\colon y^2+y = x^3+x^2-10809x-436166.
\end{align*}
The elliptic curves are not isogenous over $\overline{\Q}$.
In the proposition that follows, we prove that $(E_1, E_2)$ has a congruence obstruction to coprime reduction due to an entwinement at level $6$.

\begin{proposition}
Let $E_1$ and $E_2$ be as above. For every prime \(p\) of good reduction for both curves,
\[
\gcd(\#E_1(\F_p),\#E_2(\F_p)) \quad \text{is divisible by} \quad
\begin{cases}
2 & \text{if } p \equiv 2,6,7,8,10 \pmod{11},\\[2pt]
3 & \text{if } p \equiv 1,3,4,5,9 \pmod{11}.
\end{cases}
\]
In particular, $\#E_1(\F_p)$ and $\#E_2(\F_p)$ are not coprime for any good prime $p \nmid N_{E_1} N_{E_2}$.
\end{proposition}

\begin{proof}
The prime $p = 2$ is of bad reduction  for $E_1$ and a direct check shows that the claim holds when $p = 3$, so we assume throughout the proof that $p \geq 5$. Let \(K=\Q(\sqrt{-11})\). The squarefree parts of the minimal discriminants of \(E_1\) and \(E_2\) are both \(-11\), so \(K\subseteq \Q(E_i[2])\) for \(i=1,2\). Since the mod \(2\) Galois representations of both curves are surjective (as noted on their LMFDB \cite{lmfdb} pages), for any prime \(p \geq 5\) of good reduction that is inert in \(K\), the argument of \cite[Observation 2.3]{MR4925929} shows that \(\#E_i(\F_p)\) is even for each \(i\). Therefore, \(2\mid \gcd(\#E_1(\F_p),\#E_2(\F_p))\) whenever \( (\frac{-11}{p})=-1\). By quadratic reciprocity, this is equivalent to \(p\equiv 2,6,7,8,10 \pmod{11}\).

Now let \(p\) be a prime of good reduction for both curves that also splits in \(K\). Write
\(p\mathcal{O}_K=\mathfrak{p}\,\mathfrak{p}'\).
Then the residue field \(\mathcal{O}_K/\mathfrak{p}\) is isomorphic to \(\F_p\), and since \(p\geq 5\) is a prime of good reduction, the reduction map at \(\mathfrak{p}\) is injective on \(3\)-torsion. From the LMFDB data on the torsion in number fields we have
\(E_2(K)[3]\simeq \Z/3\Z\), so \(E_2(K)[3]\hookrightarrow E_2(\F_p)\) implies $3\mid \#E_2(\F_p)$.
It remains to prove that   \(3\mid \#E_1(\F_p)\) for the same primes \(p\). 

Let \(L=\Q(\sqrt{33})\). Again from the LMFDB, we know that \(E_1(\Q)[3]\simeq \Z/3\Z\) and  \(E_1(L)[3]\simeq \Z/3\Z\). Choose a generator \(P\in E_1(L)[3]\) and extend it to a basis \(\{P,Q\}\) of \(E_1(\overline{\Q})[3]\). With respect to this basis,  for all $\sigma \in \Gal(\overline{\Q}/\Q)$, we have
\[
\rho_{E_1,3}(\sigma)=
\begin{pmatrix}
\chi_1(\sigma) & *\\
0 & \chi_2(\sigma)
\end{pmatrix},
\]
where \(\chi_1\) is the quadratic character corresponding to \(L/\Q\) (since \(P\) is defined over \(L\)). The Weil pairing gives
\(\det\rho_{E_1,3}=\chi_{\Q(\sqrt{-3})}\), so \(\chi_2=\chi_{\Q(\sqrt{-3})}\chi_1^{-1}\) must be the quadratic character associated with $K$. For any prime \(p\geq 5\) of good reduction, the group \(E_1(\F_p)[3]\) is nontrivial if and only if \(1\) is an eigenvalue of \(\rho_{E_1,3}(\Frob_p)\). In the above upper triangular form, this holds precisely when \(\chi_1(\Frob_p)=1\) or \(\chi_2(\Frob_p)=1\), i.e., when \(p\) splits in \(L\) or in \(K\). In particular, if \(p\) splits in \(K\), then \(\chi_2(\Frob_p)=1\), so \(E_1(\F_p)[3]\neq 0\) and therefore \(3\) divides \(\#E_1(\F_p)\).

Combining the previous two paragraphs, for every good prime \(p\neq 3\) that splits in \(K\) we obtain \(3\mid \gcd(\#E_1(\F_p),\#E_2(\F_p))\).
Equivalently, for \(p\equiv 1,3,4,5,9\pmod{11}\) we have the asserted divisibility by \(3\).
Together with the inert case handled above, this shows that \(\#E_1(\F_p)\) and \(\#E_2(\F_p)\) are never coprime at any common good prime \(p\).
\end{proof}

\subsection{Serre Pair Examples}\label{S:SerrePairsII}

Jones \cite[Lemma 3.1]{MR3071819} proved that a pair \((E_1, E_2)\) of elliptic curves over \(\mathbb{Q}\) is a Serre pair if and only if both of the following conditions hold:
\begin{enumerate}
    \item[(A)] For each prime \(\ell \geq 5\), one has \(\im \rho_{E_1 \times E_2,\ell} = \Delta(\F_\ell)\), and
    \item[(B)] One has \([ \im \rho_{E_1 \times E_2,36}, \im \rho_{E_1 \times E_2,36} ] = ( \SL_2(\mathbb{Z}/36\mathbb{Z}) \cap \ker \varepsilon ) \times ( \SL_2(\mathbb{Z}/36\mathbb{Z}) \cap \ker \varepsilon )\),
\end{enumerate}
where $\varepsilon \colon \GL_2(\Z/36\Z)\to \GL_2(\Z/2\Z)\simeq S_3 \to \{\pm 1\}$ is the sign map (see \cite[p.219]{MR3557121}).

We give an alternative characterization that allows us to determine computationally whether a pair of elliptic curves is a Serre pair.

\begin{proposition}\label{criterion}
    A pair \((E_1, E_2)\) of elliptic curves over \(\mathbb{Q}\) is a Serre pair if and only if
\begin{enumerate}
    \item For each prime \(\ell \geq 5\), one has \(\im \rho_{E_1 \times E_2,\ell} = \Delta(\F_\ell)\),
    \item The curves $E_1$ and $E_2$ are Serre curves, 
    \item $[\Delta(\Z/6\Z) : \im \rho_{E_1 \times E_2,6}] \leq 2$, and
    \item $\im \rho_{E_1 \times E_2,4} = \Delta(\Z/4\Z)$.
\end{enumerate}
\end{proposition}
\begin{proof}
    We first prove the forward direction. Let $(E_1, E_2)$ be a Serre pair. As noted in \Cref{S:SerrePairs}, both $E_1$ and $E_2$ must be Serre curves, so (2) holds. We have that $m_{E_1}, m_{E_2}$ are even by \Cref{adeliclevelofSerrecurve}, so (1) holds by \Cref{SizeofSerrePair} (alternatively, (1) follows by (A) above). Further, \Cref{Atleast6} gives that $m_{E_1}, m_{E_2}\geq 6$, so (4) follows from \eqref{SizeofSerrePair}. Finally, by \Cref{m1andm2aredifferent}, at most one of $m_{E_1}, m_{E_2}$ can equal $6$, so (3) also follows by \eqref{SizeofSerrePair}. 

    We now prove the reverse direction. Assume that $(E_1,E_2)$ satisfies (1) through (4), and set
    \[
    H \coloneqq \im \rho_{E_1\times E_2,36}\subseteq \Delta(\Z/36\Z).
    \]
    Condition (1) is exactly condition (A) in Jones's criterion. By (2), both $E_1$ and $E_2$ are Serre curves. Hence, by \Cref{adeliclevelofSerrecurve} and \Cref{SerreCurveGroup}, the projections of $H$ to the two factors of $\GL_2(\Z/36\Z)\times \GL_2(\Z/36\Z)$ have commutator subgroup equal to that of $\GL_2(\Z/36\Z)$ by \cite[Corollary 2.23]{MR3350106}. Moreover, conditions (3) and (4) state that
    \[
    [\Delta(\Z/6\Z): H(6)] \le 2
    \qquad\text{and}\qquad
    H(4) =\Delta(\Z/4\Z).
    \]
    A computer search of the subgroups of $\Delta(\Z/36\Z)$ shows that every subgroup $H\subseteq \Delta(\Z/36\Z)$ satisfying these conditions has
    \[
    [H,H]
    =
    (\SL_2(\Z/36\Z)\cap \ker \varepsilon)\times
    (\SL_2(\Z/36\Z)\cap \ker \varepsilon).
    \]
    Thus condition (B) of Jones's criterion also holds, and therefore $(E_1,E_2)$ is a Serre pair.
    \end{proof}

In the repository accompanying this article, we implement the function \texttt{IsSerrePair}, which uses \Cref{criterion} to determine whether a given pair of Serre curves is a Serre pair. Condition (1) is checked using the strategy of \cite[Section 6]{MR4875533}. Condition (2) is assumed, though it can be independently verified using data from the LMFDB or Zywina's code \cite{zywina2022explicit}. Conditions (3) and (4) are checked by sampling conjugacy classes in the mod $6$ and mod $4$ images, respectively, via Frobenius elements. If \texttt{IsSerrePair} returns \texttt{true}, then it has been shown rigorously that the pair is a Serre pair, provided that both curves are Serre curves. If \texttt{IsSerrePair} returns \texttt{false}, then heuristically the pair is unlikely to be a Serre pair, but this is not shown rigorously.

\begin{example} \label{Ex:Low}
Consider the elliptic curves
\begin{align*}
    E_1 &: y^2=x^3+32x+212, \\
    E_2 &: y^2=x^3-12393x+197073.
\end{align*}
These curves have LMFDB labels \texttt{140.b1} and \texttt{34020.c1}, respectively. Both are Serre curves, with adelic levels $m_{E_1}=70$ and $m_{E_2}=210$. A computation shows that the pair $(E_1,E_2)$ satisfies the four conditions of \Cref{criterion} and thus is a Serre pair. We refer to the file \texttt{Examples.m} in the GitHub repository accompanying this article for details of the computations.

Let $f \coloneqq f_{E_1,E_2}$. Since $m_{E_1\times E_2}=\lcm(m_{E_1},m_{E_2})=210$, \Cref{SerrePairGalImage} gives
\[
G_{E_1\times E_2}(210)=G_{E_1}(210)\times_{\det}G_{E_2}(210).
\]
Using Zywina's code \cite{zywina2022explicit} to compute the groups $G_{E_1}(210)$ and $G_{E_2}(210)$, and then applying the definition of $f$ from \eqref{E:fDef}, we obtain
\[
f(210)=\frac{168823}{1358954496}.
\]
This agrees with the value determined by \Cref{calculatingfd}. Moreover, after computing $f(d)$ for each positive divisor $d$ of $210$ and applying \Cref{AlmostEulerProduct}, we find that
\[
\Cco
= \frac{\sum_{d\mid 210}\mu(d)f(d)}{\prod_{\ell\mid 210}(1-F_1(\ell))} \CcoGeneric
= \frac{5014419112}{5014521525} \CcoGeneric.
\]
This agrees with the value given by \Cref{SerrePairConstant} and shows that the lower bound in \Cref{C:SerrePairBounds} is attained. Further, for $N = 10^8$, we compute the number of good primes $p \leq N$ of coprime reduction,
\[ \pico(N) = 2250887. \]
There are a total of $5761455$ primes up to $N$. Thus, the proportion of good primes of coprime reduction among all primes up to $N$ is
\[
\frac{\pico(N)}{\pi(N)} = \frac{2250887}{5761455} \approx 0.39068.
\]
This aligns reasonably well with the value of $\Cco$ from \Cref{SerrePairConstant}.
\end{example}

\begin{example} \label{Ex:High}
Consider the elliptic curves
\begin{align*}
    E_1 &: y^2+y=x^3-81x+290, \\
    E_2 &: y^2+y=x^3-3x-2,
\end{align*}
with LMFDB labels \texttt{297.a1} and \texttt{405.a1}. These are Serre curves with adelic levels $m_{E_1}=6$ and $m_{E_2}=10$. As in \Cref{Ex:Low}, we check using \Cref{criterion} that $(E_1,E_2)$ is a Serre pair (see \texttt{Examples.m} in the accompanying GitHub repository).

Let $f \coloneqq f_{E_1,E_2}$. Since
\( m_{E_1\times E_2}=\lcm(m_{E_1},m_{E_2})=30, \)
\Cref{SerrePairGalImage} gives
\[
G_{E_1\times E_2}(30)=G_{E_1}(30)\times_{\det}G_{E_2}(30).
\]
Using Zywina's code \cite{zywina2022explicit}, we compute
\[
f(30)=\frac{5263}{884736},
\]
which agrees with the value from \Cref{calculatingfd}. It follows that
\[
\Cco
= \frac{\sum_{d\mid 30}\mu(d)f(d)}{\prod_{\ell\mid 30}(1-F_1(\ell))}\CcoGeneric
= \frac{1150648}{1118065}\CcoGeneric,
\]
which agrees with the value from \Cref{SerrePairConstant} and shows that the upper bound in \Cref{C:SerrePairBounds} is attained. Taking $N=10^8$, we compute
\[
\pico(N)=2348734.
\]
Thus,
\[
\frac{\pico(N)}{\pi(N)}=\frac{2348734}{5761455}\approx 0.40766,
\]
which aligns reasonably well with the value of $\Cco$ from \Cref{SerrePairConstant}.
\end{example}





\bibliographystyle{amsplain}
\bibliography{References}

\providecommand{\bysame}{\leavevmode\hbox to3em{\hrulefill}\thinspace}
\providecommand{\MR}{\relax\ifhmode\unskip\space\fi MR }
\providecommand{\MRhref}[2]{%
  \href{http://www.ams.org/mathscinet-getitem?mr=#1}{#2}
}
\providecommand{\href}[2]{#2}
\begin{thebibliography}{10}

\bibitem{MR3204298}
Shabnam Akhtari, Chantal David, Heekyoung Hahn, and Lola Thompson, \emph{Distribution of squarefree values of sequences associated with elliptic curves}, Women in numbers 2: research directions in number theory, Contemp. Math., vol. 606, Amer. Math. Soc., Providence, RI, 2013, pp.~171--188. \MR{3204298}

\bibitem{MR434929}
Tom~M. Apostol, \emph{Introduction to analytic number theory}, Undergraduate Texts in Mathematics, Springer-Verlag, New York-Heidelberg, 1976. \MR{434929}

\bibitem{MR2843097}
Antal Balog, Alina~Carmen Cojocaru, and Chantal David, \emph{Average twin prime conjecture for elliptic curves}, Amer. J. Math. \textbf{133} (2011), no.~5, 1179--1229. \MR{2843097}

\bibitem{Brau}
J.~Brau, \emph{Galois representations of elliptic curves and abelian entanglements}, Leiden University, 2015, Doctoral Thesis.

\bibitem{MR3447646}
Julio Brau and Nathan Jones, \emph{Elliptic curves with {$2$}-torsion contained in the {$3$}-torsion field}, Proc. Amer. Math. Soc. \textbf{144} (2016), no.~3, 925--936. \MR{3447646}

\bibitem{MR2076566}
Alina~Carmen Cojocaru, \emph{Questions about the reductions modulo primes of an elliptic curve}, Number theory, CRM Proc. Lecture Notes, vol.~36, Amer. Math. Soc., Providence, RI, 2004, pp.~61--79. \MR{2076566}

\bibitem{MR2167436}
\bysame, \emph{Reductions of an elliptic curve with almost prime orders}, Acta Arith. \textbf{119} (2005), no.~3, 265--289. \MR{2167436}

\bibitem{MR2430992}
\bysame, \emph{Square-free orders for {CM} elliptic curves modulo {$p$}}, Math. Ann. \textbf{342} (2008), no.~3, 587--615. \MR{2430992}

\bibitem{MR3349445}
Harris~B. Daniels, \emph{An infinite family of {S}erre curves}, J. Number Theory \textbf{155} (2015), 226--247. \MR{3349445}

\bibitem{MR3557121}
Harris~B. Daniels, Jeffrey Hatley, and James Ricci, \emph{Elliptic curves with maximally disjoint division fields}, Acta Arith. \textbf{175} (2016), no.~3, 211--223. \MR{3557121}

\bibitem{MR2879973}
Chantal David and Jie Wu, \emph{Almost prime values of the order of elliptic curves over finite fields}, Forum Math. \textbf{24} (2012), no.~1, 99--119. \MR{2879973}

\bibitem{dirichlet1849}
Peter Gustav~Lejeune Dirichlet, \emph{{\"U}ber die bestimmung der mittleren werthe in der zahlentheorie}, G. Lejeune Dirichlet's Werke, vol.~2, 2012, pp.~49--66.

\bibitem{MR2995149}
Tim Dokchitser and Vladimir Dokchitser, \emph{Surjectivity of mod {$2^n$} representations of elliptic curves}, Math. Z. \textbf{272} (2012), no.~3-4, 961--964. \MR{2995149}

\bibitem{MR2429913}
Ernst-Ulrich Gekeler, \emph{Statistics about elliptic curves over finite prime fields}, Manuscripta Math. \textbf{127} (2008), no.~1, 55--67. \MR{2429913}

\bibitem{CoprimeReductionsGitHub}
Asimina~S. Hamakiotes, Sung~Min Lee, Jacob Mayle, and Tian Wang, \emph{Github repository: {CoprimeReduction}}, \url{https://github.com/maylejacobj/CoprimeReduction/}, 2026.

\bibitem{MR2534114}
Nathan Jones, \emph{Averages of elliptic curve constants}, Math. Ann. \textbf{345} (2009), no.~3, 685--710. \MR{2534114}

\bibitem{MR2439422}
\bysame, \emph{A bound for the torsion conductor of a non-{CM} elliptic curve}, Proc. Amer. Math. Soc. \textbf{137} (2009), no.~1, 37--43. \MR{2439422}

\bibitem{MR2563740}
\bysame, \emph{Almost all elliptic curves are {S}erre curves}, Trans. Amer. Math. Soc. \textbf{362} (2010), no.~3, 1547--1570. \MR{2563740}

\bibitem{MR3071819}
\bysame, \emph{Pairs of elliptic curves with maximal {G}alois representations}, J. Number Theory \textbf{133} (2013), no.~10, 3381--3393. \MR{3071819}

\bibitem{MR3350106}
\bysame, \emph{{${\rm GL}_2$}-representations with maximal image}, Math. Res. Lett. \textbf{22} (2015), no.~3, 803--839. \MR{3350106}

\bibitem{MR604840}
Nicholas~M. Katz, \emph{Galois properties of torsion points on abelian varieties}, Invent. Math. \textbf{62} (1981), no.~3, 481--502. \MR{604840}

\bibitem{MR0866109}
Neal Koblitz, \emph{Elliptic curve cryptosystems}, Math. Comp. \textbf{48} (1987), no.~177, 203--209. \MR{866109}

\bibitem{MR0917870}
\bysame, \emph{Primality of the number of points on an elliptic curve over a finite field}, Pacific J. Math. \textbf{131} (1988), no.~1, 157--165. \MR{917870}

\bibitem{MR1878556}
Serge Lang, \emph{Algebra}, third ed., Graduate Texts in Mathematics, vol. 211, Springer-Verlag, New York, 2002. \MR{1878556}

\bibitem{MR568299}
Serge Lang and Hale Trotter, \emph{Frobenius distributions in {${\rm GL}\sb{2}$}-extensions}, Lecture Notes in Mathematics, Vol. 504, Springer-Verlag, Berlin-New York, 1976, Distribution of Frobenius automorphisms in ${\rm GL}\sb{2}$-extensions of the rational numbers. \MR{568299}

\bibitem{MR4925929}
Sung~Min Lee, \emph{On the average congruence class bias for cyclicity and divisibility of the groups of {$\mathbb{F}_p$}-points of elliptic curves}, J. Number Theory \textbf{278} (2026), 746--785. \MR{4925929}

\bibitem{Lee_Mayle_Wang_2025}
Sung~Min Lee, Jacob Mayle, and Tian Wang, \emph{Opposing average congruence class biases in the cyclicity and {K}oblitz conjectures for elliptic curves}, Canadian Journal of Mathematics (2025), 1–51, To appear.

\bibitem{MR916721}
H.~W. Lenstra, Jr., \emph{Factoring integers with elliptic curves}, Ann. of Math. (2) \textbf{126} (1987), no.~3, 649--673. \MR{916721}

\bibitem{lmfdb}
The {LMFDB Collaboration}, \emph{The {L}-functions and modular forms database}, \url{https://www.lmfdb.org}, 2026, [Online; accessed 23 March 2026].

\bibitem{MR3515826}
Davide Lombardo, \emph{An explicit open image theorem for products of elliptic curves}, J. Number Theory \textbf{168} (2016), 386--412. \MR{3515826}

\bibitem{MR4732685}
Jacob Mayle and Rakvi, \emph{Serre curves relative to obstructions modulo 2}, Lu{C}a{NT}: {LMFDB}, computation, and number theory, Contemp. Math., vol. 796, Amer. Math. Soc., Providence, RI, 2024, pp.~103--128. \MR{4732685}

\bibitem{MR4875533}
Jacob Mayle and Tian Wang, \emph{An effective open image theorem for products of principally polarized abelian varieties}, J. Number Theory \textbf{274} (2025), 140--179. \MR{4875533}

\bibitem{Mertens1874}
Franz Mertens, \emph{Ueber einige asymptotische gesetze der zahlentheorie.}, Journal für die reine und angewandte Mathematik \textbf{77} (1874), 289--338.

\bibitem{Miller86}
Victor Miller, \emph{Use of elliptic curves in cryptography}, Advances in Cryptology --- {CRYPTO} '85 Proceedings (Hugh~C. Williams, ed.), Springer, 1986, pp.~417--426.

\bibitem{MR1934487}
S.~Ali Miri and V.~Kumar Murty, \emph{An application of sieve methods to elliptic curves}, Progress in cryptology---{INDOCRYPT} 2001 ({C}hennai), Lecture Notes in Comput. Sci., vol. 2247, Springer, Berlin, 2001, pp.~91--98. \MR{1934487}

\bibitem{MR354514}
J.~M. Pollard, \emph{Theorems on factorization and primality testing}, Proc. Cambridge Philos. Soc. \textbf{76} (1974), 521--528. \MR{354514}

\bibitem{MR0387283}
Jean-Pierre Serre, \emph{Propri\'{e}t\'{e}s galoisiennes des points d'ordre fini des courbes elliptiques}, Invent. Math. \textbf{15} (1972), no.~4, 259--331. \MR{387283}

\bibitem{MR3223094}
\bysame, \emph{Oeuvres/{C}ollected papers. {III}. 1972--1984}, Springer Collected Works in Mathematics, Springer, Heidelberg, 2013, Reprint of the 2003 edition [of the 1986 original MR0926691]. \MR{3223094}

\bibitem{MR2140162}
J\"orn Steuding and Annegret Weng, \emph{On the number of prime divisors of the order of elliptic curves modulo {$p$}}, Acta Arith. \textbf{117} (2005), no.~4, 341--352. \MR{2140162}

\bibitem{MR2805578}
David Zywina, \emph{A refinement of {K}oblitz's conjecture}, Int. J. Number Theory \textbf{7} (2011), no.~3, 739--769. \MR{2805578}

\bibitem{zywina2022explicit}
\bysame, \emph{Explicit open images for elliptic curves over $\mathbb{Q}$}, 2022, arXiv:2206.14959.

\end{thebibliography}

\end{document}